\documentclass[a4paper,twoside]{article}
\usepackage{a4}
\usepackage{amssymb}
\usepackage{amsmath}
\usepackage{upref}
\usepackage[active]{srcltx}
\usepackage[pagebackref,colorlinks,citecolor=blue,linkcolor=blue,urlcolor=blue]{hyperref}
\usepackage[dvipsnames]{color}
\allowdisplaybreaks[2] 
\newcount\minutes \newcount\hours
\hours=\time
\divide\hours 60
\minutes=\hours
\multiply\minutes -60
\advance\minutes \time
\newcommand{\klockan}{\the\hours:{\ifnum\minutes<10 0\fi}\the\minutes}
\newcommand{\tid}{\today\ \klockan}
\newcommand{\prtid}{\smash{\raise 10mm \hbox{\LaTeX ed \tid}}}
\renewcommand{\prtid}{}
\makeatletter
\def\sectionmark#1{} 
\def\subsectionmark#1{}
\newcommand{\sectnr}{\ifnum \c@secnumdepth >\z@
                 \thesection.\hskip 1em\relax \fi}
\def\@evenhead{\footnotesize\rm\thepage\hfil\leftmark\hfil\llap{\prtid}}
\def\@oddhead{\footnotesize\rm\rlap{\prtid}\hfil\rightmark\hfil\thepage}
\def\tableofcontents{\section*{Contents} 
 \@starttoc{toc}}
\makeatother
\makeatletter
\def\@biblabel#1{#1.}
\makeatother
\makeatletter
\let\Thebibliography=\thebibliography
\renewcommand{\thebibliography}[1]{\def\@mkboth##1##2{}\Thebibliography{#1}
\addcontentsline{toc}{section}{References}
\frenchspacing 
\setlength{\@topsep}{0pt}
\setlength{\itemsep}{0pt}%
\setlength{\parskip}{0pt plus 2pt}%
}
\makeatother
\makeatletter
\def\mdots@{\mathinner.\nonscript\!.%
 \ifx\next,.\else\ifx\next;.\else\ifx\next..\else
 \nonscript\!\mathinner.\fi\fi\fi}
\let\ldots\mdots@
\let\cdots\mdots@
\let\dotso\mdots@
\let\dotsb\mdots@
\let\dotsm\mdots@
\let\dotsc\mdots@
\def\vdots{\vbox{\baselineskip2.8\p@ \lineskiplimit\z@
    \kern6\p@\hbox{.}\hbox{.}\hbox{.}\kern3\p@}}
\def\ddots{\mathinner{\mkern1mu\raise8.6\p@\vbox{\kern7\p@\hbox{.}}%
    \raise5.8\p@\hbox{.}\raise3\p@\hbox{.}\mkern1mu}}
\makeatother
\makeatletter
\let\Enumerate=\enumerate
\renewcommand{\enumerate}{\Enumerate%
\setlength{\@topsep}{0pt}
\setlength{\itemsep}{0pt}%
\setlength{\parskip}{0pt plus 1pt}%
\renewcommand{\theenumi}{\textup{(\alph{enumi})}}%
\renewcommand{\labelenumi}{\theenumi}%
}
\let\endEnumerate=\endenumerate
\renewcommand{\endenumerate}{\endEnumerate\unskip}
\makeatother
\makeatletter
\def\@seccntformat#1{\csname the#1\endcsname.\quad}
\makeatother
\newcommand{\authortitle}[2]{\author{#1}\title{#2}\markboth{#1}{#2}}
\newcommand{\auth}[2]{{#1, #2.}}
\newcommand{\art}[6]{{\sc #1, \rm #2, \it #3\/ \bf #4 \rm (#5), \mbox{#6}.}}
\newcommand{\artprep}[3]{{\sc #1, \rm #2, \rm #3.}}
\newcommand{\arttoappear}[3]{{\sc #1, \rm #2, to appear in \it #3}}
\newcommand{\book}[3]{{\sc #1, \it #2, \rm #3.}}
\newcommand{\AND}{{\rm and }}
\RequirePackage{amsthm}
\newtheoremstyle{descriptive}%
  {\topsep}   
  {\topsep}   
  {\rmfamily} 
  {}          
  {\bfseries} 
  {.}         
  { }         
  {}          
\newtheoremstyle{propositional}%
  {\topsep}   
  {\topsep}   
  {\itshape}  
  {}          
  {\bfseries} 
  {.}         
  { }         
  {}          
\newtheoremstyle{remarkstyle}%
  {\topsep}   
  {\topsep}   
  {\rmfamily}  
  {}          
  {\itshape} 
  {.}         
  { }         
  {}          
\theoremstyle{propositional}
\newtheorem{thm}{Theorem}[section]
\newtheorem{prop}[thm]{Proposition}
\newtheorem{lem}[thm]{Lemma}
\newtheorem{cor}[thm]{Corollary}
\theoremstyle{descriptive}
\newtheorem{deff}[thm]{Definition}
\newtheorem{example}[thm]{Example}
\newtheorem{remark}[thm]{Remark}
\makeatletter
\renewenvironment{proof}[1][\proofname]{\par
  \pushQED{\qed}%
  \normalfont
  \trivlist
  \item[\hskip\labelsep
        \itshape
    #1\@addpunct{.}]\ignorespaces
}{%
  \popQED\endtrivlist\@endpefalse
}
\makeatother
\def\vint{\mathop{\mathchoice%
          {\setbox0\hbox{$\displaystyle\intop$}\kern 0.22\wd0%
           \vcenter{\hrule width 0.6\wd0}\kern -0.82\wd0}%
          {\setbox0\hbox{$\textstyle\intop$}\kern 0.2\wd0%
           \vcenter{\hrule width 0.6\wd0}\kern -0.8\wd0}%
          {\setbox0\hbox{$\scriptstyle\intop$}\kern 0.2\wd0%
           \vcenter{\hrule width 0.6\wd0}\kern -0.8\wd0}%
          {\setbox0\hbox{$\scriptscriptstyle\intop$}\kern 0.2\wd0%
           \vcenter{\hrule width 0.6\wd0}\kern -0.8\wd0}}%
          \mathopen{}\int}
\RequirePackage{graphicx}\relax
\newcommand{\limplus}{{\mathchoice{\vcenter{\hbox{$\scriptstyle +$}}}
  {\vcenter{\hbox{$\scriptstyle +$}}}
  {\vcenter{\hbox{$\scriptscriptstyle +$}}}
  {\vcenter{\hbox{\scalebox{0.8}{$\scriptscriptstyle +$}}}}
}}
\newcommand{\limminus}{{\mathchoice{\vcenter{\hbox{$\scriptstyle -$}}}
  {\vcenter{\hbox{$\scriptstyle -$}}}
  {\vcenter{\hbox{$\scriptscriptstyle -$}}}
  {\vcenter{\hbox{\scalebox{0.8}{$\scriptscriptstyle -$}}}}
}}
\newcommand{\limpm}{{\mathchoice{\vcenter{\hbox{$\scriptstyle \pm$}}}
  {\vcenter{\hbox{$\scriptstyle \pm$}}}
  {\vcenter{\hbox{$\scriptscriptstyle \pm$}}}
  {\vcenter{\hbox{\scalebox{0.8}{$\scriptscriptstyle \pm$}}}}
}}
\newcommand{\limmp}{{\mathchoice{\vcenter{\hbox{$\scriptstyle \mp$}}}
  {\vcenter{\hbox{$\scriptstyle \mp$}}}
  {\vcenter{\hbox{$\scriptscriptstyle \mp$}}}
  {\vcenter{\hbox{\scalebox{0.8}{$\scriptscriptstyle \mp$}}}}
}}
{\catcode`p =12 \catcode`t =12 \gdef\eeaa#1pt{#1}}      
\def\accentadjtext#1{\setbox0\hbox{$#1$}\kern   
                \expandafter\eeaa\the\fontdimen1\textfont1 \ht0 }
\def\accentadjscript#1{\setbox0\hbox{$#1$}\kern 
                \expandafter\eeaa\the\fontdimen1\scriptfont1 \ht0 }
\def\accentadjscriptscript#1{\setbox0\hbox{$#1$}\kern   
                \expandafter\eeaa\the\fontdimen1\scriptscriptfont1 \ht0 }
\def\accentadjtextback#1{\setbox0\hbox{$#1$}\kern       
                -\expandafter\eeaa\the\fontdimen1\textfont1 \ht0 }
\def\accentadjscriptback#1{\setbox0\hbox{$#1$}\kern     
                -\expandafter\eeaa\the\fontdimen1\scriptfont1 \ht0 }
\def\accentadjscriptscriptback#1{\setbox0\hbox{$#1$}\kern 
                -\expandafter\eeaa\the\fontdimen1\scriptscriptfont1 \ht0 }
\def\itoverline#1{{\mathsurround0pt\mathchoice
        {\rlap{$\accentadjtext{\displaystyle #1}
                \accentadjtext{\vrule height1.593pt}
                \overline{\phantom{\displaystyle #1}
                \accentadjtextback{\displaystyle #1}}$}{#1}}
        {\rlap{$\accentadjtext{\textstyle #1}
                \accentadjtext{\vrule height1.593pt}
                \overline{\phantom{\textstyle #1}
                \accentadjtextback{\textstyle #1}}$}{#1}}
        {\rlap{$\accentadjscript{\scriptstyle #1}
                \accentadjscript{\vrule height1.593pt}
                \overline{\phantom{\scriptstyle #1}
                \accentadjscriptback{\scriptstyle #1}}$}{#1}}
        {\rlap{$\accentadjscriptscript{\scriptscriptstyle #1}
                \accentadjscriptscript{\vrule height1.593pt}
                \overline{\phantom{\scriptscriptstyle #1}
                \accentadjscriptscriptback{\scriptscriptstyle #1}}$}{#1}}}}
\newcommand{\setm}{\setminus}
\renewcommand{\emptyset}{\varnothing}
\DeclareMathOperator{\diam}{diam}
\DeclareMathOperator{\capp}{cap}
\newcommand{\cpXbt}{\capp_{p,\mu}^{\Xbt}}
\newcommand{\cpX}{\capp_{p,\mu}^{X}}
\newcommand{\cpXhat}{\capp_{p,\muhat}^{\Xhat}}
\newcommand{\cpG}{\capp_{p,\muhat}^{G}}
\newcommand{\ctXhat}{\capp_{t,\muhat}^{\Xhat}}
\newcommand{\cpRn}{\capp_{p,\mu}^{\R^n}}
\newcommand{\cpRt}{\capp_{p,\mu}^{\XRtp}}
\newcommand{\cpRnw}{\capp_{p,w}^{\R^n}}
\newcommand{\cpRneps}{\capp_{p,w_\eps}^{\R^n}}
\newcommand{\coneRn}{\capp_{1,\mu}^{\R^n}}
\newcommand{\cpXplus}{\capp_{p,\mu}^{\Xplus}}
\newcommand{\cpXbtp}{\capp_{p,\mu}^{\Xbtp}}
\newcommand{\cpXbtptwo}{\capp_{p,\mu}^{\XRtp}}
\newcommand{\cpXpm}{\capp_{p,\mu}^{\Xpm}}
\newcommand{\cpXminus}{\capp_{p,\mu}^{\Xminus}}

\DeclareMathOperator{\dist}{dist}
\DeclareMathOperator{\spt}{supp}
\newcommand{\supp}{\spt}
\DeclareMathOperator*{\essinf}{ess\,inf}
\newcommand{\simge}{\gtrsim}
\newcommand{\simle}{\lesssim}
\newcommand{\al}{\alpha}
\newcommand{\alp}{\alpha}
\newcommand{\be}{\beta}
\newcommand{\ga}{\gamma}
\newcommand{\eps}{\varepsilon}
\newcommand{\la}{\lambda}
\newcommand{\La}{\Lambda}
\newcommand{\om}{\omega}
\newcommand{\Om}{\Omega}
\newcommand{\clB}{\itoverline{B}}
\newcommand{\p}{{$p\mspace{1mu}$}}
\newcommand{\R}{\mathbf{R}}
\newcommand{\Xplus}{X_\limplus}  
\newcommand{\Xminus}{X_\limminus} 
\newcommand{\dplus}{d_\limplus}  
\newcommand{\dminus}{d_\limminus}  
\newcommand{\Bplus}{B_\limplus}  
\newcommand{\Bminus}{B_\limminus}  
\newcommand{\Bpm}{B_\limpm}  
\newcommand{\Bpplus}{B'_\limplus}  
\newcommand{\Bpminus}{B'_\limminus}  
\newcommand{\Br}{B_r}  
\newcommand{\Brplus}{B_{r}^{\limplus}}  
\newcommand{\Brminus}{B_{r}^{\limminus}}  
\newcommand{\Brpm}{B_{r}^{\limpm}}  
\newcommand{\Xpm}{X_\limpm} 
\newcommand{\Xbt}{X_{\R^n}}
\newcommand{\Xbtp}{X_{\R^n}^{\limplus}}
\newcommand{\Xbtm}{X_{\R^n}^{\limminus}}
\newcommand{\Xbtpm}{X_{\R^n}^{\limpm}}
\newcommand{\xpm}{x_\limpm} 
\newcommand{\xplus}{x_\limplus} 
\newcommand{\xminus}{x_\limminus} 
\newcommand{\gat}{\widetilde{\ga}}
\newcommand{\uh}{\hat{u}}
\newcommand{\gh}{\hat{g}}
\newcommand{\qmusymb}{Q}
\newcommand{\qmu}{\qmusymb_{x_0,R_0}^{X,\mu}}
\newcommand{\qmuo}{\qmusymb_{0,R_0}^{\XRtp,\mu}}
\newcommand{\qmuoRn}{\qmusymb_{0,R_0}^{\R^n,\mu}}
\newcommand{\qmuoXbtp}{\qmusymb_{0,R_0}^{\Xbtp,\mu}}
\newcommand{\XRtp}{X_{\R^2}^{\limplus}}
\newcommand{\qmup}{\qmusymb_{x_0,R_0}^{\Xplus,\mu}}
\newcommand{\qmum}{\qmusymb_{x_0,R_0}^{\Xminus,\mu}}
\newcommand{\qmupm}{\qmusymb_{x_0,R_0}^{\Xpm,\mu}}
\newcommand{\qmuhat}{\qmusymb_{x_0,R_0}^{\Xhat,\muhat}}
\newcommand{\Np}{N^{1,p}}
\newcommand{\gt}{\tilde{g}}
\newcommand{\Bt}{\widetilde{B}}
\newcommand{\xt}{\tilde{x}}
\newcommand{\Xhat}{\widehat{X}}
\newcommand{\wt}{\widetilde{w}}
\newcommand{\what}{\widehat{w}}
\newcommand{\ro}{{\rho}}
\newcommand{\muhat}{\hat{\mu}}
\newcommand{\Bhatr}{{\widehat{B}_r}}
\newcommand{\Bhat}{{\widehat{B}}}

\newcounter{saveenumi}
\numberwithin{equation}{section}
\newcommand{\eqv}{\ensuremath{\mathchoice{\quad \Longleftrightarrow \quad}
    {\Leftrightarrow}{\Leftrightarrow}{\Leftrightarrow}} }
\newcommand{\imp}{\ensuremath{\mathchoice{\quad \Longrightarrow \quad}
    {\Rightarrow}{\Rightarrow}{\Rightarrow}} }
\newenvironment{ack}{\medskip{\it Acknowledgement.}}{}

\begin{document}

\authortitle{Anders Bj\"orn, Jana Bj\"orn
    and Andreas Christensen}
{Poincar\'e inequalities and $A_p$ weights on bow-ties}

\author{
Anders Bj\"orn \\
\it\small Department of Mathematics, Link\"oping University, SE-581 83 Link\"oping, Sweden\\
\it \small anders.bjorn@liu.se, ORCID\/\textup{:} 0000-0002-9677-8321
\\
\\
Jana Bj\"orn \\
\it\small Department of Mathematics, Link\"oping University, SE-581 83 Link\"oping, Sweden\\
\it \small jana.bjorn@liu.se, ORCID\/\textup{:} 0000-0002-1238-6751
\\
\\
Andreas Christensen \\
\it\small Department of Mathematics, Link\"oping University, SE-581 83 Link\"oping, Sweden\\
\it \small andreas.christensen@liu.se, ORCID\/\textup{:} 0000-0002-0087-2488
\\
}

\date{Preliminary date, \today}

\maketitle

\noindent{\small
 {\bf Abstract}. 
A metric space $X$ is called a \emph{bow-tie} if it can be written as
$X=\Xplus \cup \Xminus$,
where $\Xplus \cap \Xminus=\{x_0\}$ and  $\Xpm \ne \{x_0\}$ are closed subsets of $X$.
We show that a doubling measure $\mu$ 
on $X$ supports a 
$(q,p)$--Poincar\'e inequality on $X$
if and only if
$X$ satisfies a quasiconvexity-type condition,
$\mu$ supports a $(q,p)$-Poincar\'e inequality on both $\Xplus$ and
$\Xminus$, and a variational \p-capacity condition holds.
This capacity condition is in turn characterized by a sharp measure decay
condition at $x_0$.

In particular, we
study the bow-tie $\Xbt$ consisting of the positive and negative hyperquadrants
in $\R^n$ equipped with a radial doubling weight and characterize the validity of
the \p-Poincar\'e inequality on $\Xbt$ in several ways.
For such weights, we also give a general formula for the capacity of annuli
around the origin.
}

\medskip

\noindent {\small \emph{Key words and phrases}:
Bow-tie, 
capacity,  
doubling measure,  
metric space, 
Muckenhoupt $A_p$-weight,
\p-admissible weight, 
Poincar\'e inequality,
radial weight, 
variational capacity.
}

\medskip

\noindent {\small \emph{Mathematics Subject Classification} (2020):
Primary:
26D10; 
 Secondary:
 30L15, 
 31C15, 
31C45, 
31E05,  
46E36. 
}

\section{Introduction}

The simplest example of a bow-tie is the union
of the first and third closed quadrants in $\R^2$,
which (when intersected with the diamond
$\{(x_1,x_2):|x_1|+|x_2|\le1\}$) resembles the shape
of a bow-tie.

In the generality of metric spaces, we say that $X=(X,d)$ is a \emph{bow-tie}
if there is a point $x_0 \in X$ which splits it in 
the form $X=\Xplus \cup \Xminus$,
where $\Xplus \cap \Xminus=\{x_0\}$ 
and $\Xpm \ne \{x_0\}$ are closed subsets of $X$.
We do not require $X$ to be complete.
$X$ is equipped with a positive complete  Borel  measure $\mu$ 
such that 
\begin{equation} \label{eq-mu-finite}
  0<\mu(B)<\infty
\quad \text{for all balls } B \subset X.
\end{equation}
We also assume that $1 \le p,q < \infty$.

Our aim is to characterize when Poincar\'e inequalities hold
on $X$ in terms of their validity on $\Xpm$.
It turns out that there are two extra conditions involved.
Here $B_r=B(x_0,r)$ is the open ball centred at $x_0$ with radius $r$,
and $\Brpm=\Br \cap \Xpm$.

\begin{thm} \label{thm-general-PI}
Assume that $\mu$ is doubling on $X$.
Then $\mu$  supports a $(q,p)$-Poincar\'e inequality on $X$ if and only if
the following conditions hold\/\textup{:}
\begin{enumerate}
\renewcommand{\theenumi}{\textup{(\roman{enumi})}}%
\item \label{c-a}
$\mu$  supports a $(q,p)$-Poincar\'e inequality on $\Xplus$ and on
$\Xminus$\textup{;}
\item \label{c-b}
there is $\Lambda$ so that
$d(\xplus,x_0) + d(x_0,\xminus) \le \Lambda d(\xplus,\xminus)$
for all $\xpm  \in \Xpm$\/\textup{;}
\item \label{c-c-both}
for all  
$0 < r < \tfrac{1}{4}\min\{\diam \Xplus, \diam \Xminus\}$, 
\begin{equation}        \label{eq-cap-cond-r-p-mu}
\cpXplus(\{x_0\}, \Brplus)  \simeq r^{-p}\mu(\Brplus) 
\quad \text{and} \quad 
\cpXminus(\{x_0\}, \Brminus)  \simeq r^{-p}\mu(\Brminus).
\end{equation}
\setcounter{saveenumi}{\value{enumi}}
\end{enumerate}
\medskip

If, in addition, $p>1$ and 
there is a locally compact open set $G \ni x_0$, 
then 
condition~\ref{c-c-both} can equivalently be replaced by
\begin{enumerate}
\renewcommand{\theenumi}{\textup{(\roman{enumi}$'$)}}%
\setcounter{enumi}{\value{saveenumi}}
\addtocounter{enumi}{-1}
\item \label{c-d}
  $p > \max\{\qmup,\qmum\}$, where $R_0=\tfrac{1}{4}\min\{\diam \Xplus, \diam \Xminus\}$
and 
\begin{equation*}  
 \qmupm := \inf \Bigl\{Q>0 :
 \frac{\mu(B_\rho^{\limpm})}{\mu(B_r^{\limpm})}
     \simge  \Bigl(\frac{\rho}{r}\Bigr)^{Q} 
     \text{ for all } 0 < \rho < r <
     R_0   
        \Bigr\}.
\end{equation*}
\end{enumerate}
\end{thm}
\medskip

Here $\cpXpm$ denotes the capacity with respect to $\Xpm$ as the
  underlying space.
Note that the dilation constants in the Poincar\'e inequalities on $X$
and $\Xpm$ are in general different, see Remark~\ref{rmk-la}.
It is well known that sets of capacity zero
cannot separate a space supporting a Poincar\'e inequality,
see \cite[Lemma~4.6]{BBbook}.
So requiring
\begin{equation} \label{eq-cp>0}
\cpXplus( \{x_0\}, \Brplus)  >0 
\quad \text{and} \quad 
\cpXminus( \{x_0\}, \Brminus)  >0
\quad \text{for small }  r>0
\end{equation}
is a necessary condition for the validity of a \p-Poincar\'e inequality
on $X$.
It has been folklore that it might also be sufficient
(see e.g.~Korte~\cite[p.~102]{Korte}).
However, 
Example~\ref{ex-3.2} 
shows that 
the range of $p>1$ so that \ref{c-c-both} holds
can be considerably smaller than the range
of $p$ for which \eqref{eq-cp>0} holds.
In this example, $d\mu=w\,dx$ is doubling and supports a 1-Poincar\'e inequality
on $\R^2$ with a radial weight~$w$ 
(and we consider the Euclidean bow-tie $\Xbt$ as in Theorem~\ref{thm-intro-radial} below).
For $p=1$ we do not know if \eqref{eq-cp>0} can hold while
\ref{c-c-both} fails.

When specializing to bow-ties in $\R^n$
equipped with radial weights we obtain the following
characterizations.
Here $x_0=(0,\ldots,0)$ is the origin.

\begin{thm} \label{thm-intro-radial}
Let $\Xbt= \Xbtp \cup  \Xbtm$, where
 \[
  \Xbtpm=\{(x_1, \ldots, x_n)\in\R^n:\pm x_j \geq 0, \ j=1,\ldots,n \}
   \quad \text{and} \quad 
   n \ge 1.
\]
Also let $d\mu=w\,dx$ be a doubling measure on $\R^n$, where
$w(x)=w(|x|)$ is a radial weight.
Then the following are equivalent\/\textup{:}
\begin{enumerate}
\item \label{f-3}
  $\mu$ supports a \p-Poincar\'e inequality on $\Xbt$\/\textup{;}
\item \label{f-1}
  $\mu$ supports a \p-Poincar\'e inequality on $\R^n$
  and
\begin{equation}   \label{eq-cap-cond-Rn-X-r-p-intro}
\cpXbtp(\{0\},B_r)  \simeq r^{-p}\mu(B_r)
\quad \text{for all } r>0;
\end{equation}
\item \label{f-2}
  $\mu$ supports a \p-Poincar\'e inequality on $\Xbtp$
  and
  \eqref{eq-cap-cond-Rn-X-r-p-intro} holds\/\textup{;}
\item \label{f-4}
  $w$ is an $A_p$-weight on $\R^n$
  and   \eqref{eq-cap-cond-Rn-X-r-p-intro} holds\/\textup{;}
\item \label{f-5}
  $\wt(\rho):=|\rho|^{n-1} w(|\rho|)$ is an $A_p$-weight on $\R$.
\end{enumerate}
\end{thm}

Unweighted $\R^n$ with $1 \le p \le n$ shows that the 
capacity condition~\eqref{eq-cap-cond-Rn-X-r-p-intro}
is not redundant in Theorem~\ref{thm-intro-radial}\,\ref{f-3}$\eqv$\ref{f-1}.
Weighted $\R^n$ shows that \eqref{eq-cap-cond-Rn-X-r-p-intro} is not
redundant in \ref{f-1}$\eqv$\ref{f-4}.
More precisely, the 
weight $w(x)=|x|^\al$ is doubling and supports a \p-Poincar\'e inequality on 
$\R^n$, with $n \ge 2$,
for all $p\ge1$ and all $\al>-n$, while it is an $A_p$ weight if and
only if $-n<\al<n(p-1)$ or $\alp=0$, see 
Heinonen--Kilpel\"ainen--Martio~\cite[p.~10]{HeKiMa} for $p>1$ and 
Proposition~\ref{prop-A1w} and Corollary~\ref{cor-wadm} for $p=1$.

Note that for radial weights,
the doubling condition for $\mu$ holds equivalently on
$\R^n$, $\Xbtp$ and $\Xbt$, see Lemma~\ref{lem-T-doubling}.
Similarly, the capacity
condition~\eqref{eq-cap-cond-Rn-X-r-p-intro} holds simultaneously for
$\cpXbtp$, $\cpXbt$ and $\cpRn$,
see Corollary~\ref{cor-anncap}.

The following theorem characterizes when the capacity 
condition~\eqref{eq-cap-cond-r-p-mu} holds.
The implications 
\[
p>\qmu \imp \eqref{eq-uQ-cp-intro} \imp p\ge\qmu
\]
appeared in 
Bj\"orn--Bj\"orn--Lehrb\"ack~\cite[Propositions~6.1 and~9.1]{BBLeh1}
(under weaker assumptions on $\mu$ and covering also the case $p=1$).
As a somewhat surprising application of Theorem~\ref{thm-general-PI}
and the deep self-improvement of Poincar\'e inequalities
due to Keith--Zhong~\cite[Theorem~1.0.1]{KeZh}, 
we now settle (in the negative) also the borderline case $p=\qmu>1$.
It is this characterization that makes it possible
  to replace \ref{c-c-both} by \ref{c-d} in the last part of 
  Theorem~\ref{thm-general-PI}.

\begin{thm} \label{thm-p>1-uQ}
Assume that $\mu$ is doubling and supports a \p-Poincar\'e inequality on $X$,
where $p>1$.
Let $0<R_0\le\tfrac14\diam X$.
Then
\begin{equation} \label{eq-uQ-cp-intro}
   \cpX(B_\rho,B_r) \simge r^{-p}\mu(\Br)
   \quad \text{for all } 0 < \rho <r < R_0   
\end{equation}
if and only if 
\begin{equation*}  
p > \qmu := \inf \Bigl\{Q >0 :
        \frac{\mu(B_\rho)}{\mu(B_r)}  \simge  \Bigl(\frac{\rho}{r}\Bigr)^{Q} 
        \text{ for all } 0 < \rho < r < R_0   
        \Bigr\}.
\end{equation*}

If, in addition, 
there is a locally compact open set $G \ni x_0$, 
then also the condition
\begin{equation} \label{eq-uQ-cp-x0intro}
   \cpX(\{x_0\},B_r) \simge r^{-p}\mu(\Br)
   \quad \text{for all } 0 < r < R_0  
\end{equation}
is equivalent to the two conditions above.
\end{thm}

When $R_0<\infty$, Lemma~\ref{lem-cp-cond-char} and 
\cite[Lemma~2.5]{BBLeh1} show that both \eqref{eq-uQ-cp-intro} and $\qmu$
are independent of the particular choice of $R_0$.
Example~\ref{ex-Q-vs-Qo} illustrates the difference between finite and infinite $R_0$.

For $p=1$ and radial weights on $\R^n$, each of \eqref{eq-uQ-cp-intro} 
and \eqref {eq-uQ-cp-x0intro} is equivalent to 
$\mu(B_\rho)/\mu(B_r)\simge\rho/r$,
i.e.\ $p=1 \ge \qmu$, see Proposition~\ref{prop-p=1-uQ}.
For other spaces and measures, the case $p=1=\qmu$ seems to be open.
When $p=1 > \qmu$, the equivalences follow from
\cite[Proposition~6.1]{BBLeh1} and Lemma~\ref{lem-cp-cond-char}.

Along the way, we obtain the
following explicit formula for the capacity 
of annuli
around the origin with respect to radial weights.
(Note that here we do not require $\mu$ to be doubling, but
\eqref{eq-mu-finite} is required as before.)

\begin{prop} \label{prop-est-cap-p-ny}
Assume that $d\mu=w\,dx$, where 
$w(x)=w(|x|)$ is a radial weight
on $\R^n$.
If $r>0$, then 
\begin{equation}   \label{eq-est-cap-p-ny}
  \cpRn(\{0\}, B_r) =
  \begin{cases}
    \displaystyle
     \biggl( \int_0^r \what(\rho)^{1/(1-p)}\,d\rho  \biggr)^{1-p},
    & \text{if } p>1, \\
    \displaystyle
     \essinf_{0<\rho<r} \what(\rho), 
    & \text{if } p = 1,
    \end{cases}
\end{equation}
where $\what(\rho):=\om_{n-1}w(\rho) \rho^{n-1}$ and
$\om_{n-1}$ is the surface area of the $(n-1)$-dimensional unit
sphere in $\R^n$ \textup{(}with $\om_0=2$\textup{)}.

Similarly, if $0 < r' < r$, then 
\begin{equation*}  
\cpRn(B_{r'},B_r) =
\cpRn(\clB_{r'},B_r) =
  \begin{cases}
    \displaystyle
     \biggl( \int_{r'}^r \what(\rho)^{1/(1-p)}\,d\rho  \biggr)^{1-p},
    & \text{if } p>1, \\
    \displaystyle
      \essinf_{r'<\rho<r} \what(\rho), 
    & \text{if } p = 1.
    \end{cases}
\end{equation*}
\end{prop}  

For $p>1$ and radial weights supporting a \p-Poincar\'e inequality on
$\R^n$, the formula for $\cpRn(B_{r'},B_r)$
was obtained 
in~\cite[Proposition~10.8]{BBLeh1}.
Our proof here is shorter, more elementary and does not
require any \p-Poincar\'e inequality.

Poincar\'e inequalities are important tools in various applications.
For example, in
Heinonen--Kilpel\"ainen--Martio~\cite{HeKiMa},
an extensive nonlinear potential theory was developed for
weighted $\R^n$ equipped with a weight $w$ supporting a \p-Poincar\'e inequality
and a doubling condition for the associated measure $d\mu=w\, dx$,
a so-called \emph{\p-admissible weight}.

Bow-ties are the simplest examples of glueing two
metric spaces together (at one point). 
Such constructions appeared in
Heinonen--Koskela~\cite[Section~6.14]{HeKo98},
where Ahlfors $Q$-regular spaces supporting Poincar\'e inequalities
were glued along various sets (not only at a single point).
Bow-ties often provide (counter)examples of metric spaces with various
interesting properties, see e.g.\
\cite[Section~8]{ABsuper},
\cite[Examples~5.6, 5.7, A.23 and~A.24]{BBbook},
\cite[Example~6.2]{BBvarcap},
\cite[p.~51]{BBnoncomp},
\cite[p.~1189]{BBLeh1},
\cite[Example~6.1]{BBLthin},
\cite[Example~4.5]{BBLehIntGreen},
\cite[Example~5.2]{BBS5},
\cite[p.~814]{JJRRS},
\cite[Remark~5.2]{KKST3},
\cite[p.~102]{Korte}
and 
\cite[Example~6.2]{KoLaSh2015}.

Shanmugalingam~\cite[Example~4.3.1]{Sh-PhD} seems to be the
first who  used the name 
\emph{bow-tie} for such examples.
Later, in~\cite{KKST3} and~\cite{KoLaSh2015}, 
they were also called \emph{Gehring bow-ties}, even though 
the example is not due to Gehring.  
Rather, the name stems from the fact that Fred Gehring (an analyst in Ann Arbor)
was always wearing a bow-tie.

The outline of the paper is as follows:
In Section~\ref{sect-prel} we discuss the necessary background from
analysis on metric spaces.
Section~\ref{sect-general} is devoted to bow-ties in metric spaces
and 
Theorem~\ref{thm-general-PI} is proved therein.
The capacity conditions~\eqref{eq-uQ-cp-intro} and~\eqref{eq-uQ-cp-x0intro}
are studied in Section~\ref{sect-uQ},
where the proof of Theorem~\ref{thm-p>1-uQ} is given.

From Section~\ref{sect-T-inv} onwards we concentrate on weights on $\R^n$
and the bow-ties $\Xbt$ (as given in Theorem~\ref{thm-intro-radial}).
In Section~\ref{sect-T-inv} we develop the theory, as far as possible,
for so-called $T$-invariant weights, while in Section~\ref{sect-radial} we turn to 
radial weights where more can be said and 
Theorem~\ref{thm-intro-radial} and Proposition~\ref{prop-est-cap-p-ny} are 
deduced.

Finally, in Section~\ref{sect-log-power} we study when 
conditions~\eqref{eq-cap-cond-r-p-mu} and~\eqref{eq-cp>0} hold
for logarithmic power weights. 
We leave it to the reader to draw conclusions using
these facts together with Theorems~\ref{thm-general-PI} and~\ref{thm-intro-radial}.

\begin{ack}
A.~B. and J.~B. 
  were supported by the Swedish Research Council, 
grants 2016-03424 and 2020-04011 resp.\ 621-2014-3974 and 2018-04106.
\end{ack}

\section{Preliminaries}
\label{sect-prel}

In this section we introduce the necessary background notation on metric
spaces and in particular on Sobolev spaces and capacities in metric spaces.
See the monographs Bj\"orn--Bj\"orn~\cite{BBbook} and
Heinonen--Koskela--Shanmugalingam--Tyson~\cite{HKSTbook}
for more extensive treatments of these topics, including
proofs of most of the results mentioned in this section.

We always assume that  $1 \le p<\infty$ 
and that $X=(X,d,\mu)$ is a metric space equipped
with a metric $d$ and a positive complete  Borel  measure $\mu$ 
such that~\eqref{eq-mu-finite} holds.
We say that $\mu$  is \emph{doubling} if
there exists a \emph{doubling constant} $C>0$ such that for all balls
$B=B(x,r):=\{y\in X: d(x,y)<r\}$ in~$X$,
\begin{equation*}
        0 < \mu(2B) \le C \mu(B) < \infty.
\end{equation*}
Here and elsewhere we let $cB=B(x,cr)$.

A \emph{curve} is a continuous mapping from an interval,
and a \emph{rectifiable} curve is a curve with finite length.
We will only consider curves which are nonconstant, compact
and rectifiable, and thus each curve can 
be parameterized by its arc length $ds$. 
Following Heinonen--Kos\-ke\-la~\cite{HeKo98},
we introduce upper gradients as follows 
(called very weak gradients in~\cite{HeKo98}).

\begin{deff} \label{deff-ug}
A Borel function $g : X \to [0,\infty]$  is an \emph{upper gradient} 
of a function $f: X \to [-\infty,\infty]$
if for all  curves  
$\gamma : [0,l_{\gamma}] \to X$,
\begin{equation*}
        |f(\gamma(0)) - f(\gamma(l_{\gamma}))| \le \int_{\gamma} g\,ds,
\end{equation*}
where the left-hand side is considered to be $\infty$ 
whenever at least one of the terms therein is infinite.
\end{deff}

The following version of Sobolev spaces on $X$
is from Shanmugalingam~\cite{Sh-rev}.

\begin{deff} \label{deff-Np}
For a measurable function $f: X\to [-\infty,\infty]$,
let 
\[
        \|f\|_{\Np(X)} = \biggl( \int_X |f|^p \, d\mu 
                + \inf_g  \int_X g^p \, d\mu \biggr)^{1/p},
\]
where the infimum is taken over all upper gradients $g$ of $f$.
The \emph{Newtonian space} on $X$ is 
\[
        \Np (X) = \{f: \|f\|_{\Np(X)} <\infty \}.
\]
\end{deff}
\medskip

The quotient space $\Np(X)/{\sim}$, where  $f \sim h$ if and only if $\|f-h\|_{\Np(X)}=0$,
is a Banach space and a lattice, see~\cite{Sh-rev}.
In this paper we assume that functions in $\Np(X)$ are defined everywhere,
not just up to an equivalence class in the corresponding function space.
This is needed for the definition of upper gradients to make sense.

\begin{deff} \label{def-PI}
Let $1 \le q < \infty$.
We say that $X$ or $\mu$ supports a \emph{$(q,p)$-Poincar\'e inequality} if
there exist constants $C>0$ and $\lambda \ge 1$
such that for all balls $B=B(x,r)$, 
all integrable functions $f$ on $X$, and all 
upper gradients $g$ of $f$, 
\begin{equation*}
        \biggl(\vint_{B} |f-f_B|^q \,d\mu\biggr)^{1/q}
        \le C r \biggl( \vint_{\lambda B} g^{p} \,d\mu \biggr)^{1/p},
\end{equation*}
where $ f_B 
 :=\vint_B f \,d\mu 
:= \int_B f\, d\mu/\mu(B)$.
If $q=1$, we usually just say \emph{\p-Poincar\'e inequality}.
\end{deff}

The Poincar\'e inequality holds equivalently 
for all measurable $f$, see
\cite[Proposition~4.13]{BBbook} and \cite[p.~50]{BBnoncomp}.
See also
\cite[Lemma~2.6]{BBLahti},
\cite[Theorem~3.2]{HaKo},
\cite[Lemma~5.15]{HeKo98}
\cite[Theorems~8.1.49 and~8.1.53]{HKSTbook}
and \cite[Theorem~2]{Keith}
for further equivalent versions.

A \emph{weight} $w$ on $\R^n$ is a nonnegative
locally integrable function
such that $d\mu=w \,dx$ is a Borel regular measure.
If $\mu$ is doubling and
supports a  \p-Poincar\'e inequality on $\R^n$,
then $w$ is called a \emph{\p-admissible weight}.
See Corollary~20.9 in \cite{HeKiMa}
(which is only in the second edition) and
Proposition~A.17 in \cite{BBbook}
for why this is equivalent to other definitions in the literature.

A weight $w$ on $\R^n$ is a
(Muckenhoupt) \emph{$A_p$-weight} if 
there exists $C>0$ such that 
\begin{equation} \label{eq-Ap-def}
\vint_B w\, dx < C 
\begin{cases} 
       \displaystyle \biggl( \vint_{B} w^{1/(1-p)}\,dx \biggr)^{1-p},
         & \text{if } 1<p<\infty, \\
      \displaystyle \essinf_B w, & \text{if }p=1,  \end{cases}  
\end{equation}
for all balls $B \subset \R^n$. 
$A_p$-weights are \p-admissible, see 
Heinonen--Kilpel\"ainen--Martio~\cite[Theorem~15.21]{HeKiMa} (for $p>1$)
and Bj\"orn~\cite[Theorem~4]{JB-Fenn} (for $p=1$).

\begin{deff} \label{deff-varcap}
Let $\Om\subset X$ be open. The \emph{variational 
\p-capacity} of $E\subset \Om$ with respect to $\Om$ is
\[
\cpX(E,\Om)
= \inf_u\int_{\Om} g^p\, d\mu,
\]
where the infimum is taken over all $u \in \Np(X)$,
such that $u=1$ in $E$ and $u=0$ on $X \setm \Om$,
and all upper gradients $g$ of $u$.
We call such a function $u$ \emph{admissible} for testing
$\cpX(E,\Om)$.  
\end{deff}

If $d\mu=w\,dx$ on $\R^n$ we 
also write
$\cpRnw(E,\Om)=\cpRn(E,\Om)$.
If $\R^n$ is equipped with a \p-admissible weight $w$,
then 
$\cpRnw$ is the usual variational capacity and $\Np(\R^n)$ and $\Np(\Om)$ are the 
refined Sobolev spaces as in
Heinonen--Kilpel\"ainen--Martio~\cite[p.~96]{HeKiMa},
see Bj\"orn--Bj\"orn~\cite[Theorem~5.1]{BBvarcap}
and~\cite[Appendix~A.2]{BBbook}.

Throughout the paper, we write $a \simle b$ if there is an implicit
constant $C>0$ such that $a \le Cb$, where $C$
may depend on fixed parameters such as $X$, $\mu$ and $p$, but
is independent of the 
essential parameters involved in $a$ and $b$. 
We also write $a \simge b$ if $b \simle a$,
and $a \simeq b$ if $a \simle b \simle a$.

\section{Bow-ties in metric spaces} 
\label{sect-general}

In this section we consider the metric space
$X=\Xplus \cup \Xminus$,
where $\Xplus \cap \Xminus=\{x_0\}$ is a fixed designated point
and $\Xpm \ne \{x_0\}$ are closed subsets of $X$.
Note that 
$X$ is connected if and only if
both $\Xplus$ and $\Xminus$ are connected.

The following condition~\ref{c-b} from Theorem~\ref{thm-general-PI}
will play a key role:
\begin{equation}    \label{eq-cond-d}
\text{There is $\Lambda$ so that}
\quad 
d(\xplus,x_0) + d(x_0,\xminus) \le \Lambda d(\xplus,\xminus)
\quad \text{for all } \xpm  \in \Xpm.
\end{equation}

Since $\Xpm\ne \{x_0\}$, we must have $\La \ge 1$.
We let $B_r=B(x_0,r)$ and $\Brpm=B_r \cap \Xpm$.

\begin{prop} \label{prop-gen-doubl}
Assume that $X$ is connected and that
\eqref{eq-cond-d} holds.
Then the measure $\mu$ is doubling on $X$ 
if and only if
the following conditions hold\/\textup{:}
\begin{enumerate}
\item \label{b-a}
$\mu$ is doubling on $\Xplus$ and on $\Xminus$\textup{;}
\item \label{b-b}
$ \mu(\Brplus) \simeq \mu(\Brminus)$
for $0<r < \min\{\diam \Xplus, \diam \Xminus\}$.
\end{enumerate}
\end{prop}

Note that $\Xplus$ and $\Xminus$ may have very
  different diameters or may be unbounded. 
This complicates the use
  of~\ref{b-b} in the following proof.

\begin{proof}
Assume that $\mu$ is doubling on $X$ and let
$0<r < \min\{\diam \Xplus, \diam \Xminus\}$.
Since $X$ is connected, the doubling condition implies that
$\mu(\{x_0\})=0$, cf.~\cite[Corollary~3.9]{BBbook}.
Since $\Xplus$ is connected we can find $\xt \in \Xplus$
  with $d(\xt,x_0)=\tfrac12 r$.
  If $y \in \Xminus$, then by~\eqref{eq-cond-d},
\begin{equation*}  
d(\xt,y) \ge \frac{1}{\Lambda} \bigl( d(\xt,x_0) + d(y,x_0)\bigr) 
\ge \frac{r}{2\Lambda}.
\end{equation*}
It thus follows that $\Bt:=B(\xt,r/2\Lambda) \subset \Brplus$.
Since $\Brminus \subset 3\Lambda \Bt$, the doubling property 
on $X$ yields
\[ 
  \mu(\Brminus) \leq \mu(3 \Lambda\Bt) \lesssim
  \mu(\Bt )
    \leq \mu(\Brplus).
\] 
The reverse inequality is shown in the same way, so \ref{b-b} holds.

To show \ref{b-a}, take any ball
$B=B(x,r)$ with $ x\in \Xplus$ and let $\Bplus=B \cap \Xplus$.
We can assume that $r\le\diam \Xplus$, since otherwise
  $\mu(2\Bplus)=\mu(\Xplus)=\mu(\Bplus)$. 
Letting $B':=B(x,r/4\La)$
and using the doubling condition on $X$, we get
\[ 
\mu (2\Bplus) \leq \mu (2B) \simle \mu (B')
= \mu(B'\cap\Xplus) + \mu(B'\cap\Xminus). 
\] 
If  $B'\cap\Xminus=\emptyset$, this immediately implies that
$\mu (2\Bplus)\simle \mu(\Bplus)$.
If  $B'\cap\Xminus \ne \emptyset$, it follows from~\eqref{eq-cond-d} 
that $d(x,x_0) <\tfrac14 r$
and so 
\[
\mu(B'\cap\Xminus) \le \mu(B_{r/2}^\limminus) \simle \mu(B_{r/2}^\limplus),
\]
where the last inequality follows from~\ref{b-b}
when $\tfrac12r<\diam\Xminus=:d$, while 
\begin{equation}   \label{eq-for-r-ge-d}
\mu(B_{r/2}^\limminus) \le \mu(\Xminus) = \mu(B_{2d}^{\limminus})
   \simle \mu(B_{d/2})
   \simeq \mu(B_{d/2}^{\limplus})
   \le \mu(B_{r/2}^{\limplus})
\end{equation}
otherwise, by~\ref{b-b} and the doubling property on $X$.
Since $B_{r/2}^{\limplus}\subset \Bplus$, this proves the doubling property
on $\Xplus$, and the argument for $\Xminus$ is similar.
Thus, \ref{b-a} holds.

Conversely, assume that~\ref{b-a} and~\ref{b-b} hold. 
Again, the connectedness of $\Xpm$ and the doubling property of $\mu$ imply that
$\mu(\{x_0\})=0$.
Let $B= B(x,r)$ be a ball in $X$ and assume without loss of generality 
that $x \in \Xplus$. 
If $d(x,x_0) \le \frac 12 r$, we see that $B_{r/2} \subset B$ and $2B \subset B_{3r}$,
and hence by the doubling property on $\Xplus$ and $\Xminus$,
\[
  \mu(2B)
  \le \mu(B_{3r}^{\limplus}) + \mu(B_{3r}^{\limminus}) 
  \simle \mu(B_{r/2}^{\limplus}) + \mu(B_{r/2}^{\limminus}) 
  \le \mu(B).
\]
Assume therefore that $d(x,x_0)>\frac 12 r$.
Using the doubling property on $\Xplus$ we have
\begin{equation}   \label{eq-split-2B-pm}
\mu(2B) = \mu(2B\cap \Xplus) + \mu(2B\cap \Xminus)
\simle \mu(B\cap \Xplus) + \mu(2B\cap \Xminus).
\end{equation}
If $2B\cap \Xminus=\emptyset$, we are done.
Otherwise, \eqref{eq-cond-d} implies that $d(x,x_0) < 2\Lambda r$ and
$2B\cap \Xminus \subset B_{4\La r}^{\limminus}$. 
Since $\diam \Xplus \ge d(x,x_0) > \frac12 r$, we see 
that
\begin{equation}  \label{eq-B-r/2-pm}
\mu(2B\cap \Xminus) \le \mu(B_{4\La r}^{\limminus})
  \simle \mu(B_{r/2}^{\limminus}) 
 \simle \mu(B_{r/2}^{\limplus}),
\end{equation}
where the last inequality follows from \ref{b-b}
and from~\eqref{eq-for-r-ge-d} with
$B_{d/2}$ replaced by $B_{d/2}^{\limminus}$.
Inserting \eqref{eq-B-r/2-pm} into \eqref{eq-split-2B-pm}, and noting that
\[
\mu(B_{r/2}^{\limplus}) \le \mu(3\La B \cap \Xplus)
  \simle \mu(B \cap \Xplus)
  \le \mu(B),
\]
shows that $\mu(2B)\simle\mu(B)$ and concludes the proof.
\end{proof}

We now turn to Theorem~\ref{thm-general-PI}.
Observe that
\[
   \cpX(\{x_0\}, \Br) = \cpXplus( \{x_0\}, \Brplus)+ \cpXminus( \{x_0\}, \Brminus).
\]

It follows from Lemma~\ref{lem-cp-cond-char} below,
that when both $\Xplus$ and $\Xminus$ are bounded,
condition~\ref{c-c-both} in Theorem~\ref{thm-general-PI}
can equivalently be replaced by the
condition
\begin{equation*}  
  \cpXplus(\{x_0\}, \Brpm)   \simeq r^{-p}\mu(\Brpm)
  \quad \text{for all }  
0 < r < \tfrac{1}{4} \diam \Xpm.
\end{equation*}
On the other hand, Example~\ref{ex-Q-vs-Qo} below shows that
this is not the case when (exactly) one of $\Xplus$ and $\Xminus$ is bounded.

The following lemma will make it possible to lift the Poincar\'e
inequality from  small to large sets when glueing $\Xplus$ and $\Xminus$
together.

\begin{lem}    \label{lem-semilocal}
{\rm (Bj\"orn--Bj\"orn~\cite[Lemma~4.11]{BBsemilocal})}
Let $q\ge1$ and $A,E\subset X$ be such that 
\[
\mu(A\cap E)\ge \theta \mu(E)
\]
for some $\theta>0$.
Also assume that for some $M\ge0$ and a measurable function~$u$,
\[ 
\|u-u_A\|_{L^q(A)} \le M
\quad \text{and} \quad
\|u-u_E\|_{L^q(E)} \le M.
\] 
Then
\[
\|u-u_{A\cup E}\|_{L^q(A\cup E)} \le 4(1+\theta^{-1/q})M.
\]
\end{lem}

\begin{proof}[Proof of Theorem~\ref{thm-general-PI}]
  Assume 
that $\mu$ supports a $(q,p)$-Poincar\'e inequality on $X$
with dilation $\la$.
We first obtain \ref{c-b}.
Let $\xpm\in \Xpm$, $r > d:=d(\xplus,\xminus)$ and $B=B(\xplus,r)$.
If $d(\xplus,x_0) \ge \la r$, then $u=\chi_{\Xplus}$ has $0$ as an upper gradient
in $\la B$ (as $\Xpm$ are closed subsets of $X$) and since 
$B\cap\Xpm\ne\emptyset$,
this contradicts the $(q,p)$-Poincar\'e inequality.
Hence $d(\xplus,x_0) < \la r$.
Letting $r \to d$ shows that $d(\xplus,x_0) \le \la d$.
Similarly, $d(x_0,\xminus) \le \la d$,
and thus
\ref{c-b}
holds with $\La=2\la$.

By \cite[Proposition~4.2]{BBbook},
$X$ is connected, and so
by Proposition~\ref{prop-gen-doubl},
  $\mu$ is doubling on $\Xplus$ and on $\Xminus$, and
\begin{equation} \label{eq-Br-comparable}
 \mu(\Brplus) \simeq \mu(\Brminus)
 \quad \text{for }0<r < \min\{\diam \Xplus, \diam \Xminus\}.
\end{equation} 
This also implies 
that $\mu(\{x_0\})=0$, see
\cite[Corollary~3.9]{BBbook}.

We next wish to obtain \ref{c-a}.
To do so, let
$\xplus \in \Xplus$,
$B = B(\xplus,r)$ and $\Bplus=B \cap \Xplus$.
We already 
know that condition~\eqref{eq-cond-d}
referred to in~\ref{c-b}  holds
with $\La=2\la$.
Next, let  $u$ be an integrable function on $\Xplus$ and $g$ be an upper
gradient of $u$. 
We may assume that $u(x_0)=0$.
Extend $u$ and $g$ as $0$ on $\Xminus \setm \{x_0\}$.
It then follows that $g$ is an upper gradient of $u$ in $X$.

In order to obtain \ref{c-a}  we will consider different cases.
We first note that when $\mu(B) \simle \mu(\Bplus)$
holds,
a simple application of the triangle inequality
implies that  
\begin{equation} \label{eq-417a}
\vint_{\Bplus} |u-u_{\Bplus}|^q \, d\mu
\le 2^q\vint_{\Bplus} |u-u_B|^q \, d\mu
\simle \vint_{B} |u-u_{B}|^q \, d\mu.
\end{equation}
Applying the $(q,p)$-Poincar\'e inequality on $X$, using that $g=0$ on $\Xminus$
 and that $\mu (\lambda \Bplus) \leq \mu(\lambda B)$, we have
\begin{equation} \label{eq-417b}
\biggl(\vint_{B} |u-u_{B}|^q \, d\mu \biggr)^{1/q}
\lesssim r \biggl( \vint_{\lambda B} g^p \, d\mu \biggr)^{1/p}
\leq  r \biggl( \vint_{\lambda \Bplus} g^p \, d\mu \biggr)^{1/p},
\end{equation}
and combining this with~\eqref{eq-417a}
yields the desired
$(q,p)$-Poincar\'e inequality for $\Bplus$.

Next we note that if $\Lambda r \leq d(x_0, \xplus)$, then for any
$\xminus \in \Xminus$ we have $d(\xplus, \xminus) \geq r$ (using \ref{c-b}),
and so $B= \Bplus$. Thus $\mu(\Bplus) =\mu(B)$  holds in this case, immediately 
yielding the $(q,p)$-Poincar\'e inequality for such balls. 
Suppose therefore that $\Lambda r > d(x_0, \xplus)$.
For any $x \in \Brplus$ we have
\[
d(x,\xplus) \leq d(x, x_0) + d(x_0, \xplus) \leq 2\Lambda r,
\]
showing that $\Brplus \subset 2\Lambda \Bplus$.
A similar estimate shows
that $B \subset 2\Lambda B_r$. 
We now consider the following cases.

\emph{Case}~1. $r \le  \min\{\diam \Xplus, \diam \Xminus\}$.
Using the
last two inclusions, together with $\mu(B_{r/2}) \simeq \mu(B_{r/2}^{\limplus})$
from \eqref{eq-Br-comparable},
and the doubling property (on $X$ followed by $\Xplus$), 
we obtain that
\[
\mu(B) \le \mu(2\La B_r)\lesssim
  \mu(B_{r/2}) \simeq \mu(B_{r/2}^{\limplus}) \le \mu(2\Lambda \Bplus)\lesssim \mu(\Bplus).
\]
Thus the $(q,p)$-Poincar\'e inequality for $\Bplus$ follows
by~\eqref{eq-417a} and~\eqref{eq-417b} in this case.

\emph{Case}~2. $\dminus:=  \diam \Xminus < r \le \diam \Xplus$.
This time the above
two inclusions,
the doubling property on 
$X$, $\Xplus$ and $\Xminus$, 
\eqref{eq-Br-comparable} applied with the radius 
$\tfrac{1}{2}\dminus <  \min\{\diam \Xplus, \diam \Xminus\}$,
and the equality $\Xminus= B_{2\dminus}^{\limminus}$
yield 
\[
  \mu(B) \leq 
  \mu(2\Lambda \Brplus) + \mu(\Xminus)
  \lesssim  \mu(\Brplus) + \mu(B_{\dminus/2}^{\limminus}) 
\lesssim  \mu(\Brplus) \le 
\mu( 2 \Lambda \Bplus) \lesssim  \mu( \Bplus),
\]
so the  $(q,p)$-Poincar\'e inequality for $\Bplus$ follows 
by~\eqref{eq-417a} and \eqref{eq-417b} for this case as well.

\emph{Case}~3.
$\dminus:=  \diam \Xminus \leq \diam \Xplus <  r$.
Then $\Bplus = \Xplus$ and $\Xminus= B_{2\dminus}^{\limminus}$. Thus
\[
\mu(B)
\le \mu(\Xplus) + \mu(\Xminus)
\lesssim  
\mu(\Xplus) + \mu(B_{\dminus/2}^{\limminus}) \simeq  \mu(\Xplus) =
\mu(\Bplus),
\]
so
we get the $(q,p)$-Poincar\'e inequality for $\Bplus$ as before.

\emph{Case}~4.
$\dplus:=  \diam \Xplus < \min\{r, \diam \Xminus\}$.
Let $\rho=\min\{r,2\dplus\}$.
In this case we 
note that
\[
\mu(\Bplus)  =
\mu(\Xplus) \geq  \mu(B^{\limplus}_{\dplus/2}) 
\simeq  \mu(B_{\dplus/2}) \gtrsim  \mu(B_{\rho}).
\]
Since
also
$\Bplus  = \Xplus \subset B_{\rho}$, we
can apply the same reasoning as in \eqref{eq-417a}--\eqref{eq-417b},
  with $B$ replaced by $B_{\rho}$, to obtain
the desired $(q,p)$-Poincar\'e inequality for $\Bplus$
with dilation $\la$, 
which thus has
been shown to hold on $\Xplus$.
Similarly $\mu$ supports a $(q,p)$-Poincar\'e inequality on $\Xminus$,
i.e.\ \ref{c-a} holds.

To verify \ref{c-c-both}, let
$0 <r< \tfrac{1}{4} \min\{\diam \Xplus, \diam \Xminus\}$
and let $u \in \Np(\Xplus)$ be a function 
admissible for testing $\cpXplus( \{x_0\}, \Brplus)$
and such that $0 \le u \le 1$.
In particular, $u(x_0)=1$ and $u(x)=0$ on $\Xplus \setm \Brplus$.
Consider the function $v = 1-u$, extended by $0$ 
to $X_\limminus$.
Let $g$ be an upper gradient of $u$ on $\Xplus$.
Then it is easily verified that $g$, extended by $0$ to $X_\limminus\setm\{x_0\}$,
is an upper gradient of $v$.
Testing the \p-Poincar\'e inequality
(which follows from the $(q,p)$-Poincar\'e inequality 
  and H\"older's inequality)
on $B_{2r}$ with $v$,
and using 
\eqref{eq-Br-comparable} and the doubling property,
shows that
\[ 
 \int_{B_{2r}} |v-v_{B_{2r}}| \, d\mu 
 \simle r \mu(B_{2r}) \biggl( \vint_{\lambda B_{2r}} g^{p} \,d\mu \biggr)^{1/p} 
 \simle r \mu(\Brplus)^{1-1/p} \biggl( \int_{\Brplus} g^{p} \,d\mu \biggr)^{1/p}.
\] 

Now, depending on whether $v_{B_{2r}}\ge\frac12$ or $v_{B_{2r}}\le\frac12$,
the left-hand side is estimated as
\[
\int_{B_{2r}} |v-v_{B_{2r}}| \, d\mu \ge \frac12 \mu(\Brminus) \simeq \mu(\Brplus)
\]
using \eqref{eq-Br-comparable},
or as 
\[
\int_{B_{2r}} |v-v_{B_{2r}}| \, d\mu \ge \frac12 \mu(B_{2r}^{\limplus} \setm B_{r}^{\limplus}) 
\simeq \mu(\Brplus),
\]
using 
the doubling property 
and the connectedness of $\Xplus$, 
cf.~\cite[Lemma~3.7]{BBbook}.

Taking infimum over all such $u$ and $g$ shows the lower bound for $\cpXplus$
in \ref{c-c-both},
while $\cpXminus$ is treated similarly.
The corresponding upper bounds
follow from Lemma~\ref{lem-cp-cond-char} below.

\medskip

Conversely, assume that \ref{c-a}--\ref{c-c-both} hold.
Then $\Xplus$ and $\Xminus$ (and so $X$) are connected,
  by \cite[Proposition~4.2]{BBbook}.
As \ref{c-b} holds, $\mu$ is doubling also on $\Xplus$ and $\Xminus$
and \eqref{eq-Br-comparable} holds,
by Proposition~\ref{prop-gen-doubl}. 
Let $B=B(x,r)$ be a ball and $\Bpm=B \cap \Xpm$.

Let $u$ be an integrable function on $X$ with upper gradient $g$.
Without loss of generality we may assume that $u(x_0)=0$.
Let $\la$ be a common dilation constant for the \p-Poincar\'e
inequalities on $\Xplus$ and $\Xminus$.

\emph{Case}~1.
\emph{$\Bplus=\emptyset$ or $\Bminus=\emptyset$.}
The cases are similar so we may assume 
the latter, i.e.\ $B \subset \Xplus$, in which case
\[
  \biggl( \vint_{B} |u-u_{B}|^q \,d\mu \biggr)^{1/q} 
 \lesssim r\biggl(\vint_{\la \Bplus} g^p \,d\mu\biggr)^{1/p}  
\le  \frac{r}{\mu(\la \Bplus)^{1/p}}
   \biggl(\int_{\la B} g^p \,d\mu\biggr)^{1/p}.  
\]
Since
$\mu(\la B) \simle \mu(B) \le \mu(\la \Bplus)$,
this concludes the proof of the Poincar\'e inequality for the ball $B$.

\emph{Case}~2a. \emph{$x=x_0$ and
\begin{equation*} 
r < r_0:=\tfrac14\min\{\diam \Xplus, \diam \Xminus\}.
\end{equation*}
}
(Here $r_0=\infty$ is allowed.)

Assumption \ref{c-c-both} and
Maz$'$ya's inequality (see \cite[Theorem~6.21]{BBbook})
for $\tfrac12\Bpm$, applied in $\Xplus$ and $\Xminus$ separately,
together with the doubling condition on $\Xpm$, 
yield
\[
 \biggl( \vint_{\Bplus} |u|^q \,d\mu  \biggr)^{1/q}
\simle \biggl(\frac{1}{\cpXplus( \{x_0\}, \Bplus)} 
\int_{\la \Bplus} g^p \,d\mu \biggr)^{1/p} 
 \simeq  r\biggl(\vint_{\la \Bplus} g^p \,d\mu\biggr)^{1/p}   
\]
and, analogously, 
\[
  \biggl( \vint_{\Bminus} |u|^q \,d\mu \biggr)^{1/q}
 \simle  r\biggl(\vint_{\la \Bminus} g^p \,d\mu\biggr)^{1/p}.
\]
Put together we have, using
also the triangle inequality,
that
\begin{align*}
\vint_{ B} |u-u_{ B}|^q \,d\mu     
&\le 2^q\vint_{B} |u|^q \,d\mu  
\lesssim \vint_{\Bplus} |u|^q \,d\mu 
          + \vint_{\Bminus} |u|^q \,d\mu   \\
& \simle r^q\biggl(\vint_{\la \Bplus} g^p \,d\mu\biggr)^{q/p}   +
r^q\biggl(\vint_{\la \Bminus} g^p \,d\mu\biggr)^{q/p}  \\
&\simle r^q \biggl(\vint_{\la B} g^p \,d\mu\biggr)^{q/p}, 
\end{align*}
where in the last step we used that
$\mu(\la \Bpm) \simeq \mu(\Bpm) \simeq \mu(B) \simeq \mu(\la B)$,
by \eqref{eq-Br-comparable}.

\emph{Case}~2b. \emph{$x=x_0$ and $r \ge r_0$.}
Then at least one of $\Xplus$ and $\Xminus$ is bounded and we may assume 
that $4r_0=\diam \Xplus \le \diam \Xminus$.
Let $B'=B_{r_0/2}$ and $B'_\limpm=B' \cap \Xpm$.
By \ref{c-a} and case~2a,
we already know that the $(q,p)$-Poincar\'e inequality holds 
for $\Bpm$ and $B'$.

Let $A=B' \cup \Bminus$.
As $\mu(B' \cap \Bminus) = \mu(\Bpminus) \simeq \mu(B')$,
by \eqref{eq-Br-comparable},
it follows from  Lemma~\ref{lem-semilocal} that
\[ 
   \int_{A} |u-u_{A}|^q \,d\mu     
   \simle  \int_{B'} |u-u_{B'}|^q \,d\mu            
    + \int_{\Bminus} |u-u_{\Bminus}|^q \,d\mu.
\] 
Similarly, $B=A \cup \Bplus$ and 
\[
\mu(A \cap \Bplus) = \mu(B'_\limplus) \simeq \mu(\Xplus) \ge \mu(\Bplus),
\]
since $\Xplus=5 B'_\limplus$ is bounded.
A second application of Lemma~\ref{lem-semilocal} then shows that 
\begin{align}   
   \int_{B} |u-u_{B}|^q \,d\mu     
  &\simle  \int_{A} |u-u_{A}|^q \,d\mu             
    +\int_{\Bplus} |u-u_{\Bplus}|^q \,d\mu 
    \label{eq-split-PI-B1-B+-2}
    \\
  & \simle \int_{B'} |u-u_{B'}|^q \,d\mu            
    + \int_{\Bminus} |u-u_{\Bminus}|^q \,d\mu
    + \int_{\Bplus} |u-u_{\Bplus}|^q \,d\mu.\nonumber
\end{align}
The $(q,p)$-Poincar\'e inequality for $B'$ (obtained in case~2a),
together with the doubling property of $\mu$,
yields
\[
\biggl( \int_{B'} |u-u_{B'}|^q \,d\mu  \biggr)^{1/q}  
\simle r_0 \mu(B')^{1/q-1/p} \biggl(\int_{\la B'} g^p \,d\mu\biggr)^{1/p}.     
\]
Similarly, 
using the $(q,p)$-Poincar\'e inequality for $\Bpm$
(from \ref{c-a}),
\[
\biggl( \int_{\Bplus} |u-u_{\Bplus}|^q \,d\mu \biggr)^{1/q}  
 \simle r_0 \mu(\Bplus)^{1/q-1/p} \biggl(\int_{\la \Bplus} g^p \,d\mu\biggr)^{1/p} 
\]
and 
\[
\biggl( \int_{\Bminus} |u-u_{\Bminus}|^q \,d\mu \biggr)^{1/q}  
 \simle r \mu(\Bminus)^{1/q-1/p} \biggl(\int_{\la \Bminus} g^p \,d\mu\biggr)^{1/p}. 
\]
Now, by the doubling property of $\mu$,
the boundedness of $\Xplus$
and 
\eqref{eq-Br-comparable},
\[
\mu(B') \simeq \mu(\Bpplus) \simeq \mu(\Bplus) \simeq \mu(\Bpminus)
\simle \mu(\Bminus) \simeq
\mu(B).
\]
Inserting these
estimates into \eqref{eq-split-PI-B1-B+-2}
and dividing by $\mu(B)\simeq\mu(\la B)$
shows the $(q,p)$-Poincar\'e inequality for $B$ when $1/q-1/p\ge0$,
i.e.\ for $q\le p$.

When $q>p$, 
the last inequality is not enough
to conclude the $(q,p)$-Poincar\'e inequality for $B$.
However, in this case, the
$(q,p)$-Poincar\'e inequality on $\Xminus$ and
\cite[Proposition~4.20]{BBbook} imply that 
\[
\frac{\mu(\Bpminus)}{\mu(\Bminus)} \simge \Bigl( \frac{r_0}{r} \Bigr)^{pq/(q-p)}
\]
and hence in~\eqref{eq-split-PI-B1-B+-2} we eventually get
\[
r_0 \mu(\Bpminus)^{1/q-1/p} \simle r \mu(\Bminus)^{1/q-1/p} 
    \simeq r \mu(B)^{1/q-1/p}.
\]

\emph{Case}~3.
\emph{Both $\Bplus$ and $\Bminus$ are nonempty and $x\ne x_0$.}
Then  $d(x,x_0) < \La r$, by \eqref{eq-cond-d}
(with $y \in B_\limmp$ if $x \in \Bpm$).
Hence 
\[
B \subset B'':=B_{(\La+1) r} \subset \la B'' \subset \la' B,
\]
where $\la' = \La + \la(\La+1)$.
Therefore, using the triangle inequality
and case~2, applied to $B''$,
we have
\begin{align*}
\vint_{B} |u-u_{B}|^q \,d\mu &\le 2^q \vint_{B} |u-u_{B''}|^q \,d\mu  
\simle \vint_{B''} |u-u_{B''}|^q \,d\mu  \\
& \simle r \biggl(\vint_{\la B''} g^p \,d\mu\biggr)^{q/p}
 \simle r \biggl(\vint_{\la' B} g^p \,d\mu\biggr)^{q/p},
\end{align*}
which shows the $(q,p)$-Poincar\'e inequality for $B$, with dilation constant  $\la'$.

\medskip

It remains to observe that the last part with \ref{c-d} now follows directly
from Theorem~\ref{thm-p>1-uQ}, applied to $\Xpm$.
Note that \ref{c-a}--\ref{c-c-both} of Theorem~\ref{thm-general-PI} are
used in the proof Theorem~\ref{thm-p>1-uQ}, but not
\ref{c-d}.
\end{proof}

\begin{remark}  \label{rmk-la}
It follows from the proof above that if $\mu$ supports a $(q,p)$-Poincar\'e
inequality on $X$ with dilation $\la$ in Theorem~\ref{thm-general-PI},
then the $(q,p)$-Poincar\'e
inequalities on $\Xpm$ also hold with dilation $\la$.
The converse is not true, as can be seen
by letting
$\Xplus=\{(t,0) : t\in \R\}$
and $\Xminus = \{t,at) : t\in \R\}$ with small $a \ne 0$,
since then there are disconnected balls in $X=\Xplus\cup\Xminus$.
\end{remark}

If $q>p$ then the $(q,p)$-Poincar\'e inequality implies that $\mu$ is doubling,
by Theorem~1 in Alvarado--Haj\l asz~\cite{AlvaradoHaj}.
Theorem~\ref{thm-general-PI} and Proposition~\ref{prop-gen-doubl}
(and the fact that only connected spaces can support Poincar\'e
  inequalities)
therefore give the following characterization of the $(q,p)$-Poincar\'e inequality,
without presupposing that $\mu$ is doubling.

\begin{cor}
If $q>p$, then
$\mu$  supports a $(q,p)$-Poincar\'e inequality on $X$ 
if and only if
conditions \ref{c-a}--\ref{c-c-both} 
in Theorem~\ref{thm-general-PI} and \eqref{eq-Br-comparable}
hold.
\end{cor}

\section{The capacity condition~\texorpdfstring{\eqref{eq-uQ-cp-intro}}{(1.4)}
and Theorem~\texorpdfstring{\ref{thm-p>1-uQ}}{1.3}}
\label{sect-uQ}

In this section we study the capacity condition 
appearing in 
Theorem~\ref{thm-general-PI}\,\ref{c-c-both}.
We will consider it within the 
metric space $X$, but the results readily apply to $\Xpm$ in 
Section~\ref{sect-general} as well.
As before $x_0 \in X$ is a designated point and $B_r=B(x_0,r)$.

\begin{lem} \label{lem-cp-cond-char}
Assume that $\mu$ is doubling on $X$.
Let $0<R_0\le\infty$ and
consider the following statements\/\textup{:}
\begin{enumerate}
\item \label{d-a}
$  \cpX( \{x_0\}, \Br)  \simeq r^{-p}\mu(\Br)$
  for all  $0< r < R_0$\textup{;}
\item  \label{d-b}
$  \cpX( \{x_0\}, \Br)  \simge r^{-p}\mu(\Br)$
  for all $0< r < R_0$\textup{;}
\item  \label{d-c}
 $ \cpX(B_\rho,B_r) \simeq r^{-p}\mu(\Br)$
  for all $0 < 2 \rho <r < R_0$\textup{;}
\item  \label{d-d}
 $ \cpX(B_\rho,B_r) \simge r^{-p}\mu(\Br)$
  for all $0 < \rho <r < R_0$.
\end{enumerate}
Then \ref{d-a} $\eqv$ \ref{d-b} $\imp$ \ref{d-c} $\eqv$ \ref{d-d}.

If there is a locally compact open set $G \ni x_0$,
then \ref{d-a}--\ref{d-d} are equivalent.

If $\mu$ supports a \p-Poincar\'e inequality on $X$,
$0 < R_1,R_2  < \infty$ 
and 
$R_1,R_2  \le \tfrac{1}{4}\diam X$ 
then each of the conditions \ref{d-a}--\ref{d-d} with 
$R_0=R_1$ is equivalent to the same condition with 
$R_0=R_2$.
\end{lem}

\begin{proof}
Clearly \ref{d-a}$\imp$\ref{d-b},
while the implications 
\ref{d-b}$\imp$\ref{d-d}  and \ref{d-c}$\imp$\ref{d-d}
follow directly from the monotonicity of $\cpX$.
The converse implications \ref{d-b}$\imp$\ref{d-a} and \ref{d-d}$\imp$\ref{d-c}
follow by 
testing the capacity with 
$u(x)=\min\{1,2(1-\dist(x,x_0)/r)_{\limplus}\}$.

Next,
assume that there is a locally compact open set $G \ni x_0$.
Theorems~1.5 and~1.9 in Eriksson-Bique--Soultanis~\cite{SEB-Soul},
applied to a compact $\clB:=\itoverline{B(x_0,r_0)}\subset G$, then 
imply that Lipschitz functions are dense in $\Np(\clB)$.
(For this, note that the doubling property of $\mu$ implies that 
$\clB$ has finite Hausdorff dimension.)
Theorem~5.29 in \cite{BBbook} then
shows that all functions in 
$\Np(\clB)$ are quasicontinuous.
In particular, if $u \in \Np(X)$, then $u$ is
quasicontinuous in $\clB$.
The proof of \cite[Theorem~6.19\,(vii)]{BBbook} then implies that
\[
\cpX({x_0},B_r) = \inf_{0<\rho<r}\cpX(B_\rho,B_r).
\]
and hence \ref{d-d} $\imp$ \ref{d-b}.

Finally,
  the last part follows directly
  from Lemma~5.5 in Bj\"orn--MacManus--Shanmugalingam~\cite{BMS}
  (or \cite[Lemma~11.22]{BBbook}).
  The proofs therein only require that
  $\mu$ is doubling and supports a \p-Poincar\'e inequality.
\end{proof}

We are now ready to prove Theorem~\ref{thm-p>1-uQ}.

\begin{proof}[Proof of Theorem~\ref{thm-p>1-uQ}]
If $p>\qmu$, Proposition~6.1 (if $R_0=\infty$) and
Theorem~1.1 (if $R_0<\infty$)
in~\cite{BBLeh1},
together with the monotonicity of $\cpX$,
show that \eqref{eq-uQ-cp-intro} holds.

Conversely, assume that \eqref{eq-uQ-cp-intro} holds
and let $\Xhat$ be the completion of $X$.
The metric $d$ extends directly to $\Xhat$ and as in
Bj\"orn--Bj\"orn~\cite{BBnoncomp} (see especially the corrigendum)
we obtain a Borel regular measure $\muhat$ on $\Xhat$ such that
$\mu(E \cap X)=\muhat(E)$ for every $\muhat$-measurable
set $E \subset \Xhat$.

It follows from Propositions~3.3 and~3.6 in \cite{BBnoncomp}
that $\muhat$ is doubling and supports a \p-Poincar\'e inequality on $\Xhat$.
It is easy to see that
\[
\cpXhat(\Bhat_\rho,\Bhat_r) \ge \cpX(B_\rho,B_r),
\]
where $\Bhatr$ is the ball in $\Xhat$ centred at $x_0$ and with radius $r$.
Thus, \eqref{eq-uQ-cp-intro} holds with $X$ replaced by $\Xhat$.
By \cite[Proposition~3.1]{BBbook}, $\Xhat$ is locally compact, 
and thus condition~\ref{d-a} in Lemma~\ref{lem-cp-cond-char}
holds with $X$ replaced by $\Xhat$.

First, assume that $R_0=\tfrac14 \diam X$ and
create a new metric space $X_0 =\Xhat_\limplus\cup \Xhat_\limminus$  made 
of two copies $\Xhat_\limpm$ of $\Xhat$ such that
$\Xhat_\limplus \cap \Xhat_\limminus = \{x_0\}$
and $d_{X_0}(\xplus,\xminus)=d(\xplus,x_0)+d(x_0,\xminus)$
if $\xpm \in \Xhat_\limpm$.
The measure $\muhat$ extends in an obvious way to $X_0$.
Since \ref{d-a} in Lemma~\ref{lem-cp-cond-char}
holds for $\Xhat$,
and   $X$ (and thus also $\Xhat$) is connected by \cite[Proposition~4.2]{BBbook},
it follows from Proposition~\ref{prop-gen-doubl} and
Theorem~\ref{thm-general-PI} (with $q=1$) that $\muhat$ is doubling and
supports a \p-Poincar\'e inequality on $X_0$.
(Note that we only use \ref{c-a}--\ref{c-c-both}
of Theorem~\ref{thm-general-PI}, without
  \ref{c-d}. The same is true below.)
By Keith--Zhong~\cite[Theorem~1.0.1]{KeZh}, there is $t<p$ such that
$\muhat$ supports a $t$-Poincar\'e inequality on $X_0$.
Hence, by 
Theorem~\ref{thm-general-PI}  again,
\[ 
\ctXhat(\Bhat_\rho,\Bhatr) \ge \ctXhat(\{x_0\},\Bhatr)  
\simeq r^{-t}\muhat(\Bhatr) 
\quad \text{for }
0 < \rho<r < \tfrac{1}{4}\diam \Xhat.
\]
Testing $\ctXhat(\Bhat_\rho,\Bhatr)$ with 
$u(x)=\min\{1,(2-\dist(x,x_0)/\rho)_{\limplus}\}$
shows that 
\[
\ctXhat(\Bhat_\rho,\Bhatr) \simle \frac{\muhat(\Bhat_\rho)}{\rho^{t}}
\quad \text{if } 0 < \rho \le \tfrac12 r.
\]
Since $\mu(B_r)=\muhat(\Bhatr)$ for all $r>0$,
comparing the last two estimates 
(and using that $\mu(B_\rho) \simeq \mu(B_r)$ when $\tfrac12 r < \rho <  r$)
implies that $p>t\ge\qmuhat=\qmu$.

Next, assume that $R_0<\tfrac14 \diam X$.
In particular, $R_0<\infty$.
If $X$ is bounded, then the equivalence between \eqref{eq-uQ-cp-intro}
and $p>\qmu$ follows from the case $R_0=\tfrac14 \diam X$, together with
the last part of Lemma~\ref{lem-cp-cond-char}.

If $X$ is unbounded, then 
there is a bounded open connected set $V \supset \Bhat_{5R_0}$ in
$\Xhat$
(since $\Xhat$ is quasiconvex, by e.g.\ \cite[Theorem~4.32]{BBbook}).
By Rajala~\cite[Theorem~1.1]{rajala}, there is
a uniform domain $G$ such that $\Bhat_{4R_0} \subset G \subset V$.
By
Bj\"orn--Shan\-mu\-ga\-lin\-gam~\cite[Lemmas~2.5, 4.2 and Theorem~4.4]{BjShJMAA},
$\muhat$ is doubling and supports a \p-Poincar\'e inequality on $G$. 
We refer to \cite{BjShJMAA} or \cite{rajala}
for the definition of uniform domains.

If $0 < \rho <r < R_0$,
then \eqref{eq-uQ-cp-intro} implies that
\[
    \cpG(\Bhat_\rho,\Bhatr)
    =  \cpXhat(\Bhat_\rho,\Bhatr)
    \ge \cpX(B_\rho,B_r)
    \simge r^{-p}\mu(\Br)
    = r^{-p}\muhat(\Bhatr),
\]
and so \eqref{eq-uQ-cp-intro} holds with $X$ replaced by $G$.
As $R_0\le\tfrac14\diam G$, an application of
the already settled bounded case
to $G$ instead of $X$, shows that $p>\qmusymb_{x_0,R_0}^{G,\muhat}=\qmu$.

The last part, concerning \eqref{eq-uQ-cp-x0intro},
follows in all cases directly from Lemma~\ref{lem-cp-cond-char}.
\end{proof}  

The following example shows that the range of $p$ in Theorem~\ref{thm-p>1-uQ}
can differ between finite and infinite $R_0$.

\begin{example}  \label{ex-Q-vs-Qo}
Let $n \ge 2$, $-n < \alp < 0$ and $d\mu=w \,dx$ on $\R^n$, where
\[
w(x)=\begin{cases}
    |x|^\alp, & \text{if }|x| \le 1, \\
    1, & \text{if }|x| \ge 1,
\end{cases}
\]
which is easily verified to be an $A_1$ weight (cf.~Proposition~\ref{prop-A1w})
and thus $1$-admissible. 
It is also rather straightforward that 
\[
\qmuoRn= \begin{cases}
         n+\al & \text{if } R_0<\infty, \\
         n & \text{if } R_0=\infty.
\end{cases}
\]
Since $n+\al<n$, this shows that the range of $p$ in Theorem~\ref{thm-p>1-uQ}
can be considerably larger for $R_0<\infty$ than for $R_0=\infty$.
\end{example}

\section{Bow-ties in \texorpdfstring{$\R^n$}{Rn}
     with \texorpdfstring{$T$}{T}-invariant weights}
\label{sect-T-inv}

From now on we  consider bow-ties $\Xbt$ in $\R^n$,
$n \ge 1$, as in Theorem~\ref{thm-intro-radial}, i.e.
\[
   \Xbt= \Xbtp \cup  \Xbtm,
   \quad \text{where }
   \Xbtpm=\{(x_1, \ldots, x_n)\in\R^n:\pm x_j \geq 0, \ j=1,\ldots,n \},
\]
are 
the positive and negative
hyperquadrants in $\R^n$.
We equip $\R^n$ with a Borel regular measure $d\mu = w \,dx$, 
where $w$ is a weight on $\R^n$.
The fixed point will be the origin $x_0=0$.

In this section we  deduce results valid for \emph{$T$-invariant weights},
i.e.\ weights such that $w \circ T= w$, where
$T:\R^n \to \Xbtp$ is given by
$T(x_1,\ldots,x_n)=(|x_1|,\ldots,|x_n|)$.
We will also say that $\mu$ is \emph{$T$-invariant} if $d\mu =w\,dx$,
  where $w$ is $T$-invariant.
In the next section we will turn to results for radial weights.

\begin{lem}\label{lem-gradient}
If $g:\Xbtp \to [0,\infty]$ is an
upper gradient of $u: \Xbtp \to [-\infty,\infty]$,
then $ \gh=g \circ T$ 
is an upper gradient of $\uh=u \circ T$ in $\R^n$.
\end{lem}

\begin{proof}
Observe first that as $g$ is a Borel function so is $\gh$.
Let $\ga:[0,l_\ga] \to \R^n$ be a curve in $\R^n$ and let
$\gat= T \circ \ga$, which is a curve in $\Xbtp$.
One can prove that 
$\gat$ is arc length parameterized,
but as we will not need this
we omit such a  proof.
It is enough 
that the arc length of every subsegment of $\ga$ is not less than the 
arc length 
of the corresponding subsegment in $\gat$,
which follows directly from the $1$-Lipschitzness of $T$.
Hence,
\[
     |\uh(\ga(0))-\uh(\ga(l_\ga))|
    = |u(\gat(0))-u(\gat(l_\ga))|
    \le  \int_{\gat} g \, ds
    \le \int_{\ga} \gh\, ds,
\]
showing that $\gh$ is an upper gradient of $\uh$. 
\end{proof}

\begin{lem}\label{lem-T} 
Assume that $\mu$ is $T$-invariant on $\R^n$.
If $r>0$, then
\[
   \cpRn(\{0\},B_r)=\inf_u \int_{B_r} g^p \, d\mu, 
\]
where the infimum is taken over all $T$-invariant 
functions $u$ such that $u(0) = 1$ and $u=0$ on $\R^n \setm  B_r$, 
and all $T$-invariant upper
gradients $g$ of $u$. 
\end{lem}

\begin{remark} \label{rmk-T}
Corresponding results for
the bow-tie $\Xbt$ and for the hyperquadrants $\Xbtpm$ are proved in the same way.
Similar identities hold also if $\{0\}$ is replaced by
$B_\rho$ or $\clB_\rho$, where $0<\rho<r$.
\end{remark}

\begin{proof}
The statement follows from Definition~\ref{deff-varcap} except
for the assertion that it suffices to
consider only $T$-invariant $u$ and $g$. 
To this end,
let $\eps >0$ and let $u$ be a function
admissible for testing
  $\cpRn(\{0\},B_r)$,
with an upper gradient $g$ such that
\[
  \cpRn(\{0\},B_r) + \eps > \int_{B_r} g^p \, d\mu 
     = \sum_{j=1}^{2^n}    \int_{B_r \cap X_j} g^p\, d\mu,
\]
where $X_1,\ldots,X_{2^n}$ are
the $2^n$ closed coordinate hyperquadrants in $\R^n$.

There is thus $k$ such that
\[
\int_{B_r \cap X_k} g^p\, d\mu
\le 2^{-n} \int_{B_r} g^p\, d\mu. 
\]
Without loss of generality, $X_k=\Xbtp$.
Next, $v:=u \circ T$ is also 
admissible for testing 
$\cpRn(\{0\},B_r)$
and $\gt:=g \circ T$ is an upper gradient of $v$,
by Lemma~\ref{lem-gradient}.
Since $w$ is $T$-invariant, we get
\[
   \int_{B_r} \gt^p\, d\mu 
= 2^n    \int_{B_r \cap \Xbtp} g^p\, d\mu 
\le    \int_{B_r} g^p \,d\mu <    \cpRn(\{0\},B_r) + \eps, 
\]
and the conclusion follows after letting $\eps \to 0$. 
\end{proof}

\begin{cor}\label{cor-anncap}
Assume that $\mu$ is $T$-invariant on $\R^n$.
If $r>0$, then
\[
\cpRn( \{0\}, B_r)  
=2^n \cpXbtp( \{0\}, B_r)
=2^{n-1} \cpXbt( \{0\}, B_r).
\]
\end{cor}

Similar identities hold also if
$\{0\}$ is replaced by $B_\rho$ or 
$\itoverline{B}_\rho$, where $0<\rho<r$.

\begin{proof} 
By Lemma~\ref{lem-T} and Remark~\ref{rmk-T} it is enough
to test all three capacities with $T$-invariant functions $u$
and $T$-invariant upper gradients $g$.
Lemma~\ref{lem-gradient} shows that
$T$-invariant upper gradients of a $T$-invariant  function
are the same with respect to $\R^n$ and $\Xbtp$,
and thus also with respect to $\Xbt$ 
(after obvious use of restrictions and extensions).
Hence the conclusions follow from the identities 
\[
 \int_{B_r} g^p \, d\mu 
  = 2^n \int_{B_r \cap \Xbtp} g^p \, d\mu 
  = 2^{n-1} \int_{B_r \cap\Xbt} g^p \, d\mu 
\quad \text{if $g$ is $T$-invariant}.
\qedhere
\]
\end{proof}

\begin{prop}\label{prop1}
Assume that $\mu$ is $T$-invariant on $\R^n$.
If  $\mu$ supports a $(q,p)$-Poincar\'e inequality  on $\R^n$,
then it also   supports a $(q,p)$-Poincar\'e inequality on $\Xbtp$
with the same dilation constant $\la$. 
\end{prop}

Note that we do not assume that $\mu$ is doubling in this result.
Thus, we cannot rely on 
Theorem~1 in Haj\l asz--Koskela~\cite{HaKo-CR}
(or \cite[Proposition~4.48]{BBbook})
to know that one may assume that $\la=1$.

Example~\ref{ex-not-Tinv} below shows that the converse implication
is not true, not even if we assume that $\mu$ is doubling.
However, in  Theorem~\ref{thm-PI-Xplus-to-Rn} below we show that
the converse implication is true for radial weights $w$ such that
$d\mu=w\,dx$ is doubling.

\begin{proof}
Let $\Bplus = B \cap \Xbtp$ 
be an arbitrary ball in $\Xbtp$, where $B=B(x,r)$ and $x \in \Xbtp$.  
For any integrable function $f$ on $\Xbtp$ we see that
\begin{equation} \label{eq-2n}
    \int_{\Bplus} f \,d\mu 
    \le \int_{B} (f \circ T) \,d\mu 
    \le 2^n \int_{\Bplus} f  \,d\mu.
\end{equation}
In particular
\begin{equation} \label{eq-Bplus-B}
  \mu(\Bplus) \simeq \mu(B).
\end{equation}

Let $u$ be an integrable function on $\Xbtp$ with 
an upper gradient $g$. 
Then $\gh = g \circ T$
is an upper gradient of $\uh = u \circ T$  by Lemma~\ref{lem-gradient}.
Using the triangle inequality, 
\eqref{eq-Bplus-B} and \eqref{eq-2n}, we get that
\begin{align*}
\biggl(\vint_{\Bplus} |u-u_{\Bplus}|^q \, d\mu\biggr)^{1/q}
  &\le 2 \biggl(\vint_{\Bplus} |u-\uh_{B}|^q \, d\mu\biggr)^{1/q}
  \simeq  \biggl(\vint_{B} |\uh-\uh_{B}|^q \, d\mu\biggr)^{1/q} 
\\
  & \simle r \biggl(\vint_{\la B} \gh^p \, d\mu \biggr)^{1/p}
 \le 2^n r \biggl(\vint_{\la \Bplus} g^p \, d\mu \biggr)^{1/p}.
\qedhere
\end{align*}
\end{proof}

\begin{lem} \label{lem-T-doubling}
Assume that $\mu$ is $T$-invariant on $\R^n$.
Then $\mu$ is doubling on $\R^n$ if and only if it
is doubling on $\Xbtp$, which in turn happens if and only if it is
doubling on $\Xbt$.
\end{lem}  

\begin{proof}
  As $\mu(B(x,r))=\mu(B(T(x),r))$ this follows directly
  from \eqref{eq-Bplus-B} deduced in the proof
  of Proposition~\ref{prop1}.
\end{proof}

\begin{prop}\label{prop-w=woT}
Assume that $\mu$ is a doubling $T$-invariant  measure on $\R^n$.
Then
$\mu$
supports a $(q,p)$-Poincar\'e inequality
on $\Xbt$
if and only if
it
supports a $(q,p)$-Poincar\'e inequality
on $\Xbtp$
and  
\begin{equation}   \label{eq-cap-cond-Rn-X-r-p}
\cpXbtp( \{0\}, B_r)  \simeq r^{-p}\mu(B_r)
\quad \text{for all } r>0.
\end{equation}
\end{prop}

\begin{proof}
This is a direct consequence of Theorem~\ref{thm-general-PI}.
\end{proof}

By \cite[Proposition~4.48]{BBbook} the dilation for the
$(q,p)$-Poincar\'e inequality
  on $\Xbt$  can be chosen to satisfy
$\la \le \sqrt{2}$,
  while on $\Xbtp$ it can be required to be $1$.
Lemma~\ref{lem-T-doubling} shows
that the doubling condition for $\mu$ holds equivalently on
$\R^n$, $\Xbtp$ and $\Xbt$.
Similarly, by Corollary~\ref{cor-anncap}, the capacity
condition~\eqref{eq-cap-cond-Rn-X-r-p} holds simultaneously for
$\cpXbtp$, $\cpXbt$ and $\cpRn$.

Next, we 
give an example showing that the converse implication
  in Proposition~\ref{prop1}  can fail
for doubling measures given by $T$-invariant weights.

\begin{example} \label{ex-not-Tinv}
Let $1 \le p <\infty$ and $n \ge 2$.
We construct a $T$-invariant weight on $\R^n$ which 
is \p-admissible on $\Xbtp$ but not on $\R^n$.
In particular, it shows that
Theorem~\ref{thm-PI-Xplus-to-Rn} below does not hold for $T$-invariant weights,
and it is essential that we only consider radial weights therein.

Let $d\mu=w\,dx$, where $w(x_1,x_2,\ldots,x_n)=|x_1|^\alp$ with $\alp > p-1$,
and let $d\nu(x)=|x|^\al\,dx$
be a measure on $Y=[0,\infty)$.

Then $\mu$ is doubling on $\R^n$ and $\nu$ on $Y$.
By Chua--Wheeden~\cite[Theorem~1.4]{ChuaWheeden},
$\nu$ supports a $1$-Poincar\'e inequality for each
interval $(a,b)\subset Y$ with $\la=1$ and the optimal constant
\begin{align*}
     C
 &=\frac{4}{(b-a)\nu(a,b)} 
     \biggl\|\frac{\nu((a,x))\nu((x,b))}{x^\al} \biggr\|_{L^\infty(a,b)} 
    \le \frac{4}{(b-a)} 
     \biggl\|\frac{\nu((a,x))}{x^\al} \biggr\|_{L^\infty(a,b)} \\\
 & = \frac{4}{(b-a)} 
     \biggl\|\frac{x^{1+\alp}-a^{1+\alp}}{(1+\alp)x^\al} \biggr\|_{L^\infty(a,b)} 
  = \frac{4}{(b-a)} 
     \biggl\|\frac{(x-a) \xi_x^\alp}{x^\alp} \biggr\|_{L^\infty(a,b)} 
  \le  4,
\end{align*}
where $\xi_x \in(a,x)$ comes from the mean-value theorem.
As this holds for all intervals $(a,b)\subset Y$, 
$\nu$ supports a $1$-Poincar\'e inequality on $Y$
with constant $4$.
It then follows from Bj\"orn--Bj\"orn~\cite[Theorem~3 and Remark~4]{BBtensor},
that $\mu$ supports a $1$-Poincar\'e inequality on $\Xbtp=Y^n$.

To show that $\mu$ is not \p-admissible on $\R^n$, let
$K=\{0\} \times [0,1]^{n-1}$ and $u_t(x)=(1-t\dist(x,K))_\limplus$, $t \ge 1$.
It then follows that the Sobolev capacity (see \cite{BBbook} or \cite{HeKiMa})
\[
    C_{p,\mu}^{\R^n}(K) \le \|u_t\|_{\Np(\R^n,\mu)}^p 
      \le \mu(\supp u_t) + \int_{\supp u_t} t^p \, d\mu
    \simle  t^{-\al-1} + t^{p-\al-1},
\]
which tends to $0$ as $t \to \infty$.
As sets of capacity zero cannot separate open sets in spaces supporting
Poincar\'e inequalities (see \cite[Lemma~4.6]{BBbook}), $\mu$
does not support any \p-Poincar\'e inequality on $\R^n$.
\end{example}

\section{Bow-ties in \texorpdfstring{$\R^n$}{Rn} with radial weights }
\label{sect-radial}

We say that a weight $w$ is a \emph{radial weight} if
there is a function  $w: [0,\infty) \to [0,\infty]$ such that
$w(x)=w(|x|)$.
By abuse of notation we denote both functions by $w$, and treat
other radial functions similarly.
We will also say that $\mu$ is \emph{radial} if $d\mu =w\,dx$,
where $w$ is radial.

\begin{lem}\label{lem-radial} 
Assume that $\mu$ is radial on $\R^n$.
If $r>0$, then
\[
\cpRn(\{0\},B_r)=\inf_u \int_{B_r} g^p \, d\mu, 
\]
where the infimum is taken over all radial
functions $u$ such that $u(0) = 1$ 
and $u=0$ on $\R^n \setm  B_r$, 
and all radial upper gradients $g$ of $u$. 
\end{lem}

This was shown in the proof of
Proposition~10.8 in
Bj\"orn--Bj\"orn--Lehrb\"ack~\cite{BBLeh1}.
(The argument is valid also for $p=1$,
as well as for $n=1$.)
The corresponding result is proved in the same way for
the bow-tie $\Xbt$, as well as for $\Xbtp$. 
Similar identities hold also if $\{0\}$ is replaced by
$B_\rho$ or $\clB_\rho$, where $0<\rho<r$.

For doubling radial weights, Proposition~\ref{prop1} admits a converse:

\begin{thm}   \label{thm-PI-Xplus-to-Rn}
Assume that 
$\mu$ is a doubling radial measure on $\R^n$.
Then
the following are equivalent\/\textup{:}
\begin{enumerate}
\item \label{d-Xp}
$\mu$ supports a $(q,p)$-Poincar\'e  inequality on $\Xbtp$\/\textup{;}
\item \label{d-Xp-1}
  $\mu$ supports a $(q,p)$-Poincar\'e  inequality on $\Xbtp$ 
  with dilation $\la=1$\/\textup{;}
\item \label{d-Rn}
$\mu$ supports a $(q,p)$-Poincar\'e  inequality on $\R^n$\/\textup{;}
\item \label{d-Rn-1}
$\mu$ supports a $(q,p)$-Poincar\'e  inequality on $\R^n$
  with dilation $\la=1$.
\end{enumerate}
\end{thm}

The equivalences \ref{d-Xp}$\eqv$\ref{d-Xp-1}
and \ref{d-Rn}$\eqv$\ref{d-Rn-1} are well known and due to
Haj\l asz--Koskela~\cite[Theorem~1]{HaKo-CR},
but it will be convenient to have them included here.
Recall that the doubling condition for the radial measure 
$\mu$ holds equivalently on
$\R^n$, $\Xbtp$ and $\Xbt$, by Lemma~\ref{lem-T-doubling}.

Theorem~\ref{thm-PI-Xplus-to-Rn}
can directly be combined with Proposition~\ref{prop-w=woT}
to give a characterization for the $(q,p)$-Poincar\'e  inequality on $\Xbt$.

\begin{proof}
\ref{d-Xp}$\eqv$\ref{d-Xp-1}
and \ref{d-Rn}$\eqv$\ref{d-Rn-1}
As mentioned above, these equivalences follow from
Haj\l asz--Koskela~\cite[Theorem~1]{HaKo-CR}
(or \cite[Proposition~4.48]{BBbook}).

\ref{d-Rn}$\imp$\ref{d-Xp}
This follows from Proposition~\ref{prop1}.

\ref{d-Xp-1}$\imp$\ref{d-Rn}
By symmetry, any 
sector congruent to $\Xbtp$ in $\R^n$,
with vertex at the origin,
  also supports a $(q,p)$-Poincar\'e inequality
  with the same constants.
In particular, this is true for any of the $2^n$ coordinate hyperquadrants.

Let $B=B(z,r)$ be a ball in $\R^n$ and $u$ be an integrable function on
$\R^n$ with an upper gradient $g$. 
By symmetry, we can assume that $z_1=z_2=\ldots=z_n\ge0$.
If $r\le z_1$ then $B\subset \Xbtp$ and the $(q,p)$-Poincar\'e
inequality holds for $B$. 
Assume therefore that $r> z_1 = |z|/\sqrt{n}$.
In particular, 
\begin{equation}  \label{eq-inclusions-balls}
B \subset B(0,(1+\sqrt{n})r) \subset (1+2\sqrt{n})B.
\end{equation}
It therefore suffices to show that the $(q,p)$-Poincar\'e inequality holds on
every ball $B_r$ centred at the origin, cf.\ \eqref{eq-417a}--\eqref{eq-417b}.
Consider the  coordinate hyperquadrant
\[ 
X'=\{(x_1, \ldots, x_n):x_1\le0 \text{ and } x_j \geq 0,\ j=2,\ldots,n \},
\]
neighbouring to $\Xbtp$.
Moreover, let 
\[ 
X''=\{(x_1, \ldots, x_n):x_2\ge|x_1| \text{ and } x_j \geq 0,\ j=3,\ldots,n \}
\]
be a congruent sector intersecting both $\Xbtp$ and $X'$.
Note that, by the choice of $X''$ and because $w$ is radial,
\[
\mu(B_r\cap\Xbtp\cap X'') = \mu(B_r\cap X'\cap X'') 
= \tfrac12 \mu(B_r\cap X'') 
= \tfrac12 \mu(B_r\cap X')
= \tfrac12 \mu(B_r\cap \Xbtp).
\]
By assumption, 
\begin{equation*}
\biggl(\int_{B_r\cap\Xbtp} |u-u_{B_r\cap\Xbtp}|^q \, d\mu\biggr)^{1/q}
  \le  C(r) \biggl(\int_{B_r\cap\Xbtp} g^p \, d\mu \biggr)^{1/p},
\end{equation*}
where $C(r)=C r\mu(B_r\cap\Xbtp)^{1/q-1/p}$.
Similar $(q,p)$-Poincar\'e inequalities with the same $C(r)$ hold also with $\Xbtp$
replaced by $X'$ and $X''$.
Lemma~\ref{lem-semilocal}, applied first to $B_r\cap\Xbtp$ and
$B_r\cap X''$, and then to $B_r\cap(\Xbtp\cup X'')$ and $B_r\cap X'$,
now implies that 
\begin{align*}
\biggl(\int_{B_r\cap(\Xbtp\cup X')} |u-u_{B_r\cap(\Xbtp\cup X')}|^q \, d\mu\biggr)^{1/q} \\
  & \kern -5em 
\le C(r)(4(1+2^{1/q}))^2 \biggl(\int_{B_r\cap(\Xbtp\cup X')} g^p \, d\mu \biggr)^{1/p}.
\end{align*}
Continuing in this way and aggregating further coordinate
hyperquadrants
using Lemma~\ref{lem-semilocal}, we obtain 
\begin{equation*}
\biggl(\int_{B_r} |u-u_{B_r}|^q \, d\mu\biggr)^{1/q}
  \le C(r)(4(1+2^{1/q}))^{2^{n+1}} 
     \biggl(\int_{B_r} g^p \, d\mu \biggr)^{1/p}.
\end{equation*}
Since $C(r)\simeq r\mu(B_r)^{1/q-1/p}$ and in view of the 
inclusions~\eqref{eq-inclusions-balls}, this implies
a $(q,p)$-Poincar\'e inequality with a dilation constant at most
$\la=1+2\sqrt{n}$ for every ball in $\R^n$,
which concludes the proof.
\end{proof}

We are now ready to prove the explicit formula for 
the capacity 
of annuli
around the origin with respect to radial weights.

\begin{proof}[Proof of Proposition~\ref{prop-est-cap-p-ny}]
By Lemma~\ref{lem-radial}, it suffices to test the capacity by radial
functions $u(x)=u(|x|)$ and their radial upper gradients $g(x)=g(|x|)$.
Since $g$ is radial, it follows that $\int_{0}^rg\,d\rho\ge1$.

Consider first the formula~\eqref{eq-est-cap-p-ny} for $p=1$.
Using
spherical coordinates,
\[
\int_{B_r} g\,d\mu 
=  \int_{0}^r g(\rho) \what(\rho) \,d\rho
\ge  \essinf_{(0,r)} \what.
\]
Taking infimum over all
such $u$ and their radial upper gradients $g$
shows one inequality in~\eqref{eq-est-cap-p-ny} for $p=1$.
To prove the reverse
inequality, let 
\[
A_\eps := \Bigl\{ \rho\in (0,r): \what(\rho) 
            \le \essinf_{(0,r)} (\what+\eps) \Bigr\}, \quad \eps>0,
\]
and consider the function
\[
u(x) = 1 -\frac{1}{|A_\eps|} \int_{0}^{\min\{r,|x|\}} \chi_{A_\eps}(\rho)\,d\rho,
\]
where $|A_\eps|$ stands for the Lebesgue measure of $A_\eps$.
Then $u$ is Lipschitz on $\R^n$, admissible in the definition of
$\coneRn(\{0\},B_r)$ 
and has $g(x)=|A_\eps|^{-1} \chi_{A_\eps}(|x|)$ as an upper gradient.
(We may assume that $\what$ is a Borel function, and thus
  $A_\eps$ is a Borel set.)
Since $\what\le \essinf_{(0,r)}(\what+\eps)$ on $A_\eps$, it then follows that
\[
\coneRn(\{0\},B_r) 
\le \frac{1}{|A_\eps|} \int_{A_\eps} \what(\rho)\,d\rho
\le \essinf_{(0,r)} ( \what +\eps).
\]
Letting $\eps\to0$ concludes the proof of~\eqref{eq-est-cap-p-ny} for $p=1$.

Next, we turn to the formula~\eqref{eq-est-cap-p-ny} for $p>1$.
Since $w$ is locally integrable, $\what(\rho)<\infty$ for a.e.\ $\rho$.
In this case, we get by H\"older's inequality
(with appropriate interpretation when $\what(\rho)=0$),
\begin{equation*} 
  1    \le \int_0^r g \what^{1/p} \what^{-1/p}\, d\rho
   \le \biggl( \int_{B_r} g^p \,d\mu\biggr)^{1/p}
       \biggl( \int_0^r \what^{1/(1-p)} \,d\rho\biggr)^{(p-1)/p},
\end{equation*}
and thus
\[
         \biggl( \int_0^r \what^{1/(1-p)} \,d\rho\biggr)^{1-p}
    \le \int_{B_r} g^p \,d\mu.
\]
Taking infimum over all
such $u$ and their radial upper gradients $g$
shows one inequality in~\eqref{eq-est-cap-p-ny} for $p>1$.

Conversely, let 
$w_\eps(\rho)=w(\rho)+ \eps$ and $\what_\eps(\rho)=\om_{n-1}w_\eps(\rho)\rho^{n-1}$.
It then follows
that
\[
0<  \int_0^r \what_\eps^{p/(1-p)} \,d\rho < \infty.
\]
Let now $g_\eps(x)=g_\eps(|x|)=c \what_\eps(|x|)^{1/(1-p)}$, where $c$ is chosen
so that $\int_0^r g_\eps \,d\rho=1$.
Also let 
\[
   u_\eps(x)=1-\int_0^{\min\{r,|x|\}} g_\eps(t) \,dt.
\]
We may again assume that $\what$ is a Borel function.
Then $g_\eps$ is an upper gradient of $u_\eps$ in $\R^n$.
Considering when equality holds in H\"older's inequality,
we have
\begin{equation*} 
   1    = \int_0^r g_\eps \what_\eps^{1/p} \what_\eps^{-1/p}\, d\rho
   = \biggl( \int_{B_r} g_\eps^p w_\eps\,dx\biggr)^{1/p}
       \biggl( \int_0^r \what_\eps^{1/(1-p)} \,d\rho\biggr)^{(p-1)/p}
\end{equation*}
which shows that
\[
   \cpRn(\{0\}, B_r)
   \le \cpRneps(\{0\}, B_r)
   \le \biggl( \int_0^r \what_\eps(\rho)^{1/(1-p)}\,d\rho  \biggr)^{1-p}.
\]
Letting $\eps \to 0$ concludes
the proof of the other inequality in~\eqref{eq-est-cap-p-ny} for $p>1$.
The last formula in the proposition is shown similarly.
\end{proof}

\begin{thm}   \label{thm-PI-bowtie-imp-Ap-R}
Assume that $\mu$ is a doubling radial measure on $\R^n$.
If $\mu$
supports a \p-Poincar\'e  inequality on $\Xbt$,
then the weight $\wt(\rho):=|\rho|^{n-1} w(|\rho|)$ is an $A_p$-weight on $\R$.
\end{thm}

When $n=1$ this recovers
  Theorem~2 in Bj\"orn--Buckley--Keith~\cite{BBK-Ap}
  in the special case when $\mu$ is radial.

\begin{proof}
Proposition~\ref{prop-w=woT} shows that $\mu$
supports a \p-Poincar\'e  inequality
on $\Xbtp$
and that the capacity condition~\eqref{eq-cap-cond-Rn-X-r-p}
is satisfied.
Theorem~\ref{thm-PI-Xplus-to-Rn} implies that $\mu$
also supports a \p-Poincar\'e  inequality on $\R^n$.

Hence, by Proposition~\ref{prop-est-cap-p-ny},
the capacity condition~\eqref{eq-cap-cond-Rn-X-r-p} for $\cpRn( \{0\}, B_r)$ 
is for $p>1$ equivalent to 
\[
\biggl( \int_0^r \wt^{1/(1-p)}\,d\rho  \biggr)^{1-p} 
\simeq \frac{1}{r^{p}} \int_{B_r}w\,dx 
\simeq \frac{1}{r^{p}} \int_0^r  \wt \,d\rho.
\]
For $p=1$ we instead get by
Proposition~\ref{prop-est-cap-p-ny}
that
\eqref{eq-cap-cond-Rn-X-r-p} is equivalent to 
\[
\essinf_{(0,r)} \wt\simeq \frac1r \int_0^r  \wt\,d\rho.
\]
In both cases, this gives the $A_p$-condition for $\wt$ and intervals
in $\R$ centred at the origin.

If $I=(t-r,t+r)$ is an interval in $\R$ such that $r\ge\tfrac14|t|$, then
$I\subset (-5r,5r)$ and hence also
(when $p>1$)
\[
\int_I  \wt\,d\rho \biggl( \int_I \wt^{1/(1-p)}\,d\rho  \biggr)^{1-p} 
\le \int_{-5r}^{5r}  \wt\,d\rho \biggl( \int_{-5r}^{5r} \wt^{1/(1-p)}
       \,d\rho \biggr)^{p-1} 
\simeq (5r)^p.
\]
Thus, the $A_p$-condition for $\wt$ holds for $I$ when $p>1$, and
since 
\[
\essinf_{I} \wt \ge \essinf_{(0,5r)} \wt 
\simeq \frac{1}{5r} \int_0^{5r} \wt\,d\rho
\simge \vint_I  \wt\,d\rho,
\] 
also when $p=1$.

Now assume that $p>1$,
$I=(t-r,t+r)$ is an interval in $\R$ and that $r<\tfrac14|t|$.
We can assume that $t>0$.
Let $\eps>0$ be arbitrary and consider the function
\begin{equation}   \label{eq-def-u-by-int}
u(x) = \int_{t-r}^{\min\{t+r,|x|\}}  (w(\rho) +\eps)^{1/(1-p)} \,d\rho.
\end{equation}
Then $u$ is Lipschitz on $\R^n$, vanishes on $B(0,t-r)\subset\R^n$ and 
\begin{equation}   \label{eq-u=a}
u \equiv \int_{t-r}^{t+r} (w(\rho)+\eps)^{1/(1-p)} \,d\rho
=: a_{\eps,t,r} \quad \text{on } \R^n \setm B(0,t+r).
\end{equation}
We may assume that $w$ is a Borel function. Then clearly,
\[
g=(w(|x|)+\eps)^{1/(1-p)}\chi_{(t-r,t+r)}(|x|)
\] 
is an upper gradient of $u$.
We shall apply the \p-Poincar\'e inequality to $u$ in the ball
$B(z,2r)$, where $z=(t,0,\ldots,0)$.
Note that either
\begin{equation*}   
|u-u_{B(z,2r)}| = |u_{B(z,2r)}| \ge \tfrac12 a_{\eps,t,r} \quad \text{in } B(z,2r)\cap B(0,t-r)
\end{equation*}
or
\begin{equation*}
|u-u_{B(z,2r)}| = |a_{\eps,t,r}-u_{B(z,2r)}| \ge \tfrac12 a_{\eps,t,r} 
\quad \text{in } B(z,2r)\setm B(0,t+r).
\end{equation*}
Moreover, because $r<\tfrac14 t$ and $\mu$ is doubling,
\[
\mu\bigl(B(z,2r)\cap B(0,t-r)\bigr) 
\simeq \mu\bigl(B(z,2r)\setm B(0,t+r)\bigr) \simeq \mu(B(z,2r)).
\]
The \p-Poincar\'e inequality on $B(z,2r)$, together with the spherical
coordinates, 
then implies that 
\begin{align}
a_{\eps,t,r}^p &\simle r^p \vint_{B(z,2r)} g^p \, d\mu
  = r^p  \vint_{B(z,2r)} (w(|x|)+\eps)^{p/(1-p)} \chi_{(t-r,t+r)}(|x|) w(|x|) \,dx  \nonumber \\
  &\simle r^p  \frac{r^{n-1}}{\mu(B(z,2r))} \int_{t-r}^{t+r} (w(\rho)+\eps)^{1/(1-p)} 
        \,d\rho 
  =  r^p  \frac{r^{n-1}}{\mu(B(z,2r))} a_{\eps,t,r}.    
\label{eq-PI-gives-est-a}
\end{align}
Since by spherical coordinates,
\begin{equation}   \label{eq-est-mu(B(z,2r))}
  \mu(B(z,2r))\simeq
  \mu(B(z,r))\simeq
  r^{n-1} \int_{t-r}^{t+r} w \,d\rho,
\end{equation}
we conclude from~\eqref{eq-u=a} and~\eqref{eq-PI-gives-est-a} that 
\[
\int_{t-r}^{t+r} w(\rho)\,d\rho
\biggl( \int_{t-r}^{t+r} (w(\rho)+\eps)^{1/(1-p)} \,d\rho
\biggr)^{p-1}
\simle r^p.
\]
Letting $\eps\to0$ and noting that $\wt(\rho)\simeq t^{n-1}w(\rho)$
for all $\rho\in I$ gives the $A_p$-condition~\eqref{eq-Ap-def} 
for $I=(t-r,t+r)$ also when $r<\tfrac14|t|$, and concludes the proof
when $p>1$.

For $p=1$, replace~\eqref{eq-def-u-by-int} by the function
\[
u(x) = \int_{0}^{|x|} \chi_{A_\eps}(\rho) \,d\rho,
\]
where 
\[
A_\eps := \biggl\{ \rho\in I: w(\rho)  \le \essinf_{I} (w+\eps) \biggr\}.
\]
Since $\chi_{A_\eps}$ is an upper gradient of $u$
(because $w$ can be assumed to be a Borel function),
arguments similar
to those for $p>1$ lead to 
\[
|A_\eps| \simle r \vint_{B(z,2r)} \chi_{A_\eps}(|x|) w(|x|) \,dx  
  \simle \frac{r^{n}|A_\eps|\essinf_{I} (w+\eps)}{\mu(B(z,2r))}  
\]
Finally, using
\eqref{eq-est-mu(B(z,2r))},
dividing by $|A_\eps|>0$,
and letting $\eps\to0$ concludes the proof also for $p=1$.
\end{proof}

The following result is the last link in proving Theorem~\ref{thm-intro-radial}.

\begin{thm}   \label{thm-PI-Ap-R-imp-Rn+bowtie}
Assume that $w$ is radial on $\R^n$.
If the weight $\wt(\rho):=|\rho|^{n-1} w(|\rho|)$ is an
$A_p$-weight on $\R$, then $w$ is an $A_p$-weight on $\R^n$ and 
the capacity condition~\eqref{eq-cap-cond-Rn-X-r-p} holds.
\end{thm}

\begin{proof}
Assume that $\wt$ is an $A_p$-weight on $\R$ and let
$B=B(z,r)\subset\R^n$ be a ball.
If $r\le\tfrac14|z|$ then let $I=(|z|-r,|z|+r)$ be the corresponding interval in $\R$.
Note that $\rho\simeq|z|$ for all $\ro\in I$.
Using spherical coordinates we therefore have
\begin{equation}    \label{eq-int-w-int-wt}
\int_B w\,dx \simle |z|^{1-n} r^{n-1} \int_I \wt(\rho) \,d\rho
\end{equation}
and similarly (when $p>1$),
\[
\int_B w^{1/(1-p)}\,dx \simle |z|^{(n-1)/(p-1)} r^{n-1} \int_I \wt^{1/(1-p)}(\rho) \,d\rho.
\]
Multiplying the last two estimates, cancelling $|z|$ and using the $A_p$-condition 
for $\wt$ on $I$ yields
\begin{equation}   \label{eq-Ap-B}
\int_B w\,dx \biggl( \int_B w^{1/(1-p)}\,dx \biggr)^{p-1} \simle r^{p(n-1)}r^p = r^{np},
\end{equation}
i.e.\ the $A_p$-condition for $w$ holds on $B$ when $p>1$.
For $p=1$ we instead combine~\eqref{eq-int-w-int-wt} with the
$A_1$-condition for $\wt$ and with
\[
\essinf_I \wt \simle
|z|^{n-1} \essinf_B w.
\]

If $r>\tfrac14|z|$, then $B\subset B(0,5r)$ and it therefore remains
to consider balls centred at the origin.
For such balls, spherical coordinates show that
(when $p>1$)
\begin{align*}
\int_{B_r} w^{1/(1-p)}\,dx &\simeq \int_0^r w(\rho)^{1/(1-p)} \rho^{n-1} \,d\rho
\le r^{(n-1)p/(p-1)} \int_0^r \wt(\rho)^{1/(1-p)} \,d\rho, \\
\int_{B_r} w\,dx &\simeq \int_0^r \wt(\rho) \,d\rho \quad \text{and} \quad 
\essinf_{(0,r)} \wt \le r^{n-1} \essinf_{B_r} w.
\end{align*}
This and the $A_p$-condition for $\wt$ immediately give~\eqref{eq-Ap-B}, 
as well as its analogue for $p=1$, also for $B_r$.
Thus $w$ is an $A_p$-weight on $\R^n$.

Finally, it remains
to note that by
Proposition~\ref{prop-est-cap-p-ny},
the lower bound in
the capacity condition~\eqref{eq-cap-cond-Rn-X-r-p} 
is equivalent to the assumed $A_p$-condition for $\wt$ and $B_r$.
The upper bound in~\eqref{eq-cap-cond-Rn-X-r-p}  then follows
from Lemma~\ref{lem-cp-cond-char}.
\end{proof}

We are now ready to deduce Theorem~\ref{thm-intro-radial}.

\begin{proof}[Proof of Theorem~\ref{thm-intro-radial}]
  \ref{f-1}$\imp$\ref{f-2}
  This follows from Proposition~\ref{prop1}.

  \ref{f-2}$\eqv$\ref{f-3}
  This follows from Theorem~\ref{thm-general-PI} or
  Proposition~\ref{prop-w=woT}.

  \ref{f-3}$\imp$\ref{f-5}
  This follows from Theorem~\ref{thm-PI-bowtie-imp-Ap-R}.

  \ref{f-5}$\imp$\ref{f-4}
  This follows from Theorem~\ref{thm-PI-Ap-R-imp-Rn+bowtie}.

  \ref{f-4}$\imp$\ref{f-1}
This follows from  Heinonen--Kilpel\"ainen--Martio~\cite[Theorem~15.21]{HeKiMa} (for $p>1$)
and Bj\"orn~\cite[Theorem~4]{JB-Fenn} (for $p=1$).
\end{proof}

We end this section with the following result for $p=1$, 
which is in contrast to the case $p>1$ in Theorem~\ref{thm-p>1-uQ}.

\begin{prop} \label{prop-p=1-uQ}
Assume that
$\mu$  is a doubling radial  measure on $\R^n$ which supports a
$1$-Poincar\'e inequality. 
Let $0<R_0\le\infty$.
Then
the following conditions are equivalent\/\textup{:}
\begin{enumerate}    
\item \label{k-1}
$\displaystyle \frac{\mu(B_\rho)}{\mu(B_r)} \simge \frac{\rho}{r}
   \  \text{for all } 0 < \rho <r <R_0$,
i.e.\ $\qmuoRn=\qmuoXbtp \le 1;$
\item \label{k-2}  
$ 
   \coneRn(B_\rho,B_r) \simge r^{-1}\mu(\Br)
   \ \text{for all } 0 < \rho <r<R_0;
$ 
\item \label{k-3}  
$ 
   \coneRn(\{x_0\},B_r) \simge r^{-1}\mu(\Br)
   \ \text{for all } 0<r<R_0.
$ 
\end{enumerate}
\end{prop}

\begin{proof}
\ref{k-2}$\eqv$\ref{k-3} This follows directly
    from Lemma~\ref{lem-cp-cond-char}.

\ref{k-2}$\imp$\ref{k-1}
For $\rho\le\tfrac12r$, this follows from the simple estimate
$\coneRn(B_\rho,B_r) \simle \rho^{-1}\mu(B_\rho)$, which is obtained by testing
the capacity with  
\[
u(x)=\min\{1,(2-|x|/\rho)_{\limplus}\}. 
\]
The monotonicity of $\mu$ takes care of $\tfrac12r<\rho<r$.

\ref{k-1}$\imp$\ref{k-2}
Let $\wt(t):=w(t) t^{n-1}$,
$0 < \eps < r$ and
\[
u(x)=\min\biggl\{\frac{r+\eps-|x|}{2\eps},1\biggr\}_\limplus.
\]
Then, by the doubling property of $\mu$
and the $1$-Poincar\'e inequality,
\[
\mu(B_{r-\eps}) \simle
  \int_{B_{2r}} |u - u_{B_{2r}}|\, d\mu
  \simle
    r \int_{r-\eps}^{r+\eps}  \frac{1}{2\eps} \wt(t) \, dt.
\]
Letting $\eps \to 0$ shows that
\[
\frac{\mu(B_r)}{r} \simle \wt(r)
  \quad \text{for a.e. } r >0.
\]
On the other hand, by Proposition~\ref{prop-est-cap-p-ny}
and using \ref{k-1},
\begin{equation*} 
  \coneRn(B_{\rho},B_r)
\simeq \essinf_{\rho<t<r} \wt(t)
\simge \essinf_{\rho<t<r} \frac{\mu(B_t)}{t}
\simge \frac{\mu(B_r)}{r}
\quad 
   \text{if } 0 <  \rho <r.
\qedhere
\end{equation*}
\end{proof}

We are now going to give an example (of a radial $1$-admissible weight)
such that the range of $p$ for which the capacity
condition~\eqref{eq-cap-cond-Rn-X-r-p} holds is
considerably smaller than the range for which 
$\cpXbtp(\{0\},\Br)>0$.  

\begin{example} \label{ex-3.2}
We will follow the construction
in Bj\"orn--Bj\"orn--Lehrb\"ack~\cite[Example~3.2]{BBLeh1}.
Let $\alp_k=2^{-2^k}$ and
 $\be_k= \alp_k^{3/2}=2^{-3\cdot 2^{k-1}}$, $k=0,1,2,\ldots$\,.
Note that $\alp_{k+1}= \alp_k^2$. 
Consider the measure $d\mu=w(|y|)\,dy$ on $\R^2$, where
\[
      w(\rho)=\begin{cases}
                     \alp_{k+1}, & \text{if } \alp_{k+1} \le \rho \le \be_k, 
                     \ k=0,1,2,\ldots,\\
                     \rho^2/ \alp_{k}, & \text{if } \be_k \le \rho \le \alp_{k}, 
                     \ k=0,1,2,\ldots,\\
                     \rho, & \text{if } \rho \ge \tfrac{1}{2}.
        \end{cases}
\] 
It follows from Example~3.2 in~\cite{BBLeh1} that $w$ is $1$-admissible,
$\qmuo=4$ and 
\[
\text{$\mu(B_r)\simge r^{s}$ for all $0<r<1$}
\quad \text{if and only if}
\quad s\ge\tfrac{10}{3}.
\]
We thus conclude from Theorem~\ref{thm-p>1-uQ} that
the capacity condition~\eqref{eq-cap-cond-Rn-X-r-p} holds
if and only if $p>4$.
Proposition~8.2 in~\cite{BBLeh1} shows that 
\[    
\cpXbtptwo(\{0\},B_r)>0 \quad \text{if } p >\tfrac{10}{3} \text{ and } r>0,
\]
so for $p\in(\tfrac{10}{3},4]$, the weighted bow-tie does not support a \p-Poincar\'e 
inequality, even though the glueing point has positive capacity.
A straightforward calculation, using the formula in Proposition~\ref{prop-est-cap-p-ny}, 
shows that $\capp_{\frac{10}{3},\mu}^{\XRtp}(\{0\},B_r)=0$.

Given $b>a>1$,
Example~3.4 in~\cite{BBLeh1} provides similar radial weights, such that 
$\qmuo=b$ and hence the capacity condition~\eqref{eq-cap-cond-Rn-X-r-p} holds
exactly when $p>b$,   
while $\cpRt(\{0\},B_r)>0$ for all $p>a$.
\end{example}

\section{Logarithmic power weights on \texorpdfstring{$\R^n$}{Rn}}
\label{sect-log-power}

Let from now on $\phi(\rho) = \max \{1, -\log \rho \}$.
The following integral estimate is
straightforward, but a bit tedious.

\begin{lem}\label{lem-intprop}
It is true that
\[ 
\int_0^r \rho^{a-1} \phi(\rho)^b \, d\rho \simeq   
\begin{cases}
r^{a}\phi(r)^b, & \text{if } a>0 \text{ and } b \in \R, \\
\phi(r)^{b+1},  & \text{if }  a=0,\  b<-1 \text{ and } r \le 1, \\
1+\log r,  & \text{if }  a=0,\  b<-1 \text{ and } r > 1, \\
\infty,  & \text{if } a <0, \text{ or } a=0 \text{ and } b \ge -1. \\
\end{cases} 
\] 
\end{lem}

\begin{prop}\label{prop-A1w}
Let  $w(|x|)=|x|^{\alpha} \phi(|x|)^{\beta}$
be a weight on $\R^n$,  where 
$\alp > -n$ and $\be \in \R$.
Then $w$ is an $A_1$-weight if and only if
$\al<0$ or $\al=0 \le \be $.
\end{prop}

\begin{proof}
It follows directly from the definition that a continuous 
$A_1$-weight cannot vanish at any point. 
Hence, $\al>0$ as well as purely logarithmic weights with $\be<0=\al$ cannot 
give  $A_1$-weights. 

Conversely assume that $\al<0$ or $\al=0 \le \be $.
Note that $\phi(r') \simeq \phi(r)$ if $r \le r' \le 2r$,
and thus also $w(r') \simeq w(r)$ if $r \le r' \le 2r$.
Moreover $w$ is approximately decreasing in the sense that
\[ 
w(r) \simge w(r')
\quad \text{whenever } 0 < r < r'.
\] 
Let $B=B(x_0, r)$ be an arbitrary ball. 
If $r \le  \frac12 |{x_0}|$, then 
since $w$ is approximately decreasing we get  that
\[ 
 \vint_B w(x) \, dx \simle w(|x_0|-r) \simeq w(|x_0|+r)
\simeq \essinf_{B} w.
\]
On the other hand, if $r> \frac12 |{x_0}|$, then
$B \subset B_{3r}$.
Hence using 
Lemma~\ref{lem-intprop} we get that
\[ 
 \vint_B w \, dx \simle  
 \frac{1}{r^{n}} \int_{B_{3r}} w \, dx 
\simeq \frac{1}{r^{n}} \int_0^{3r} \rho^{\alp+n-1} \phi(\rho)^\be \, d\rho
\simeq r^{\alpha} \phi(3r)^{\beta}. 
\] 
Moreover, 
\[ 
\essinf_{B} w \geq \essinf_{B_{3r}} w 
\simeq w(3r) \simeq r^{\alpha} \phi(3r)^{\beta},
\]
which shows that the $A_1$-condition 
is satisfied
for all balls.
\end{proof}

Using Proposition~\ref{prop-A1w} we obtain the following result
which is well known when $\be=0$, see 
Heinonen--Kilpel\"ainen--Martio~\cite[p.~10]{HeKiMa}.

\begin{cor}\label{cor-wadm}
Let $w(|x|)=|x|^{\alpha} \phi(|x|)^{\beta}$
be a weight on $\R^n$, $n \geq 2$, where $\alpha > -n$ 
and $\beta \in \R$. 
Then $w$ is $1$-admissible.
\end{cor}

For $n=1$ this is false, since 
any $1$-admissible weight on $\R$ is an $A_1$-weight
by Theorem~2 in Bj\"orn--Buckley--Keith~\cite{BBK-Ap}.

\begin{proof}
  For $\alp>1-n$ this follows from
\cite[Remark~10.6]{BBLeh1},
while for $-n < \alp \le 1-n$ it follows
from Proposition~\ref{prop-A1w}.
\end{proof}

We next characterize when these types of weights satisfy
$\cpRn(\{0\}, B_r)>0$ resp.\ $\cpRn(\{0\}, B_R)\simeq r^{-p}\mu(B_r)$.
This again shows that the former condition can hold without the latter one,
at least when $p>1$, cf.\ Example~\ref{ex-3.2}.

\begin{prop}\label{prop-Cpp}
Let $w(|x|)=|x|^{\alpha} \phi(|x|)^{\beta}$
be a weight on $\R^n$, $n \geq 1
$,
with $\alpha > -n$ and $ \beta \in \R$. 
Then $\cpRn(\{0\}, B_r)>0$ for some, or equivalently all, $r>0$
if and only if one of the following conditions holds\textup{:}
\begin{enumerate}
\item \label{al-1}
$\alpha < p-n$, 
\item \label{al-2}
$\alpha = p-n$  and $\beta > p-1$,
\item \label{al-3}
$p=1$, $\alpha=1-n$ and $\be\ge0$.
\end{enumerate}
Moreover, with a fixed $0<R_0\le\infty$, we have
\[
\cpRn(\{0\}, B_r)\simge \frac{\mu(B_r)}{r^p}  \quad \text{for all } 0<r<R_0
\]
if and only if \ref{al-1} or \ref{al-3} holds.
\end{prop}

\begin{proof}
Since $\al+n>0$, Lemma~\ref{lem-intprop} shows that 
$r^{-p}\mu(B_r)\simeq r^{\al+n-p}\phi(r)^\be$ for all $\be\in\R$ and $r>0$.
To estimate $\cpRn(\{0\}, B_r)$ assume first that $p>1$.
It follows from 
Proposition~\ref{prop-est-cap-p-ny}
and Lemma~\ref{lem-intprop} that
\begin{align*}
\cpRn(\{0\}, B_r) 
&\simeq  \biggl(\int_0^{r} \rho^{(\alpha+n-1)/(1-p)} \phi(\rho)^{\beta/(1-p)} \, d\rho\biggr)^{1-p}
\\
&\simeq   
\begin{cases}
r^{\alpha + n - p}\phi(r)^{\beta}, & \text{if } \alpha < p-n, \\
\phi(r)^{\beta+1-p}, & \text{if }\alpha = p-n, \ \beta > p-1 \text{ and }  r\le1, \\
(1+\log r)^{1-p}, & \text{if }\alpha = p-n, \ \beta > p-1 \text{ and }  r>1, \\
0,  & \text{otherwise}, \\
\end{cases} 
\end{align*}
which proves the statement
when $p>1$.

Assume for the rest of the proof that $p=1$.
Proposition~\ref{prop-est-cap-p-ny} implies that for all $r>0$,
\[
\cpRn(\{0\}, B_r) \simeq \essinf_{0<\rho<r} \rho^{\al+n-1}\phi(r)^\be.
\]
Thus, $\cpRn(\{0\}, B_r)>0$ if and only if either $\al<1-n$ or 
$\al=1-n$ and $\be\ge0$.
In this case also
\[
\cpRn(\{0\}, B_r) \simeq r^{\al+n-1}\phi(r)^\be \simeq \frac{\mu(B_r)}{r}, 
\]
which concludes the proof.
\end{proof}

\end{document}